\documentclass{elsarticle}%
\usepackage{amsfonts}
\usepackage{amsthm}
\usepackage{amsmath}
\usepackage{amssymb}
\usepackage{mathtools}
\usepackage{xcolor}
\usepackage{soul}
\usepackage{mathrsfs}
\usepackage{tabularx}
\usepackage[top=1in, bottom=1in, left=1in, right=1in]{geometry}
\usepackage{graphicx}%
\providecommand{\U}[1]{\protect\rule{.1in}{.1in}}
\newtheorem{theorem}{Theorem}[section]
\newtheorem{lemma}{Lemma}[section]

\newtheorem{assumption}{Assumption}[section]
\theoremstyle{definition}
\newtheorem{definition}{Definition}[section]
\theoremstyle{remark}

\numberwithin{equation}{section}
\begin{document}
	\begin{frontmatter}
		
		\title{Solvability and Optimal Controls of Impulsive Stochastic  Systems  in Hilbert Spaces}
		
		%%this line removes the date, but space is still left for it;
		%if used, remove the \vspace{-1cm}
		%\date{}
		
		%this gives the date in the form Mon 30 Jan 2012, 8:57pm;
		%if used, retain the \vspace{-1cm}
		%\date{\shortdayofweekname{\day}{\month}{\year}{ }\mydate\today}

		\author[1]{Javad A. Asadzade}
		\ead{javad.asadzade@emu.edu.tr}
            \author[1,2]{Nazim I. Mahmudov}
		\ead{nazim.mahmudov@emu.edu.tr}
		\cortext[cor1]{Corresponding author}

            \address[1]{Department of Mathematics, Eastern Mediterranean University, Mersin 10, 99628, T.R. North Cyprus, Turkey}
	\address[2]{Research Center of Econophysics, Azerbaijan State University of Economics (UNEC), Istiqlaliyyat Str. 6, Baku, 1001,  Azerbaijan}
	
		% Latex won't make the title unless told:
		%\maketitle
		
		%%to remove the space left for date, use:

		\begin{abstract}
           This paper investigates the solvability and optimal control of a class of impulsive stochastic differential equations (SDEs) within a Hilbert space setting. First, we establish the existence and uniqueness of mild solutions for the proposed impulsive stochastic system, leveraging fixed-point theorems and appropriate analytical techniques. Next, we identify and derive the necessary conditions for the existence of optimal control pairs, ensuring the feasibility and effectiveness of the control solutions. Finally, to validate and demonstrate the practical applicability of our theoretical findings, we provide a detailed example showcasing the utility of the results in real-world scenarios.
		\end{abstract}

  
	\end{frontmatter}
\section{Introduction}\label{sec1}

The impulsive nature of many evolutionary processes is characterized by brief disturbances at particular times. Such dynamic impulsive systems are prevalent in areas such as robotics, telecommunications, computer science, neural networks, biological systems, population dynamics, genetics, and artificial intelligence. For a comprehensive analysis of impulsive systems, see \cite{41}.

SDEs are mathematical models used to describe systems influenced by random phenomena or uncertainties. They extend ordinary differential equations by incorporating stochastic processes, typically modeled using Wiener processes (Brownian motion). SDEs are widely applied in fields like physics, finance, biology, and engineering to capture dynamic systems with inherent randomness. Impulsive SDEs are an extension of SDEs that incorporate sudden, discontinuous changes (impulses) in the system's state at specific moments in time. These models are particularly useful in describing real-world systems that experience abrupt perturbations or interventions, such as shocks, switching dynamics, or resets, alongside continuous random fluctuations.The study of SDEs under impulsive effects has garnered significant attention in recent years due to their  range of applications in fields such as finance, engineering, and biological systems. These systems are characterized by sudden changes in the state of the system at specific moments, which can model real-world phenomena more accurately than continuous processes alone. The main advantage of incorporating impulsive effects is that they provide a more realistic representation of systems that undergo abrupt changes, making the models more applicable to real-world scenarios. Impulsive SDEs have diverse applications across various disciplines:
\begin{itemize}
    \item \textbf{Engineering:}
    Control systems with sudden adjustments (e.g., robotic movements or circuit switching).

    \item \textbf{Biology:}
    Population dynamics with sudden environmental events or medical interventions.

    \item \textbf{Finance:}
    Modeling markets with abrupt shocks or policy changes.

    \item \textbf{Physics:}
    Systems undergoing random motion interrupted by external impulses (e.g., particle collisions).

    \item \textbf{Ecology:}
    Species dynamics with seasonal effects or catastrophic events.
\end{itemize}
As mentioned above, impulsive systems play a crucial role in modeling real-world phenomena that involve abrupt changes or discontinuities. They find applications in fields such as control theory, numerical analysis, and complex networks. For instance, studies in \cite{2,4,46} addressed the controllability of linear first-order, fractional, and semilinear impulsive systems, emphasizing their importance in understanding and influencing dynamic systems with sudden interventions.

In control theory, impulsive systems have been used to solve optimal control problems for SDEs with random impulses (\cite{43,44}). In \cite{43}, a random compensation function was introduced into the performance index, leading to a new Hamilton-Jacobi-Bellman (HJB) equation. The existence/uniqueness of its viscosity solution were proven. Similarly, \cite{44} defined a refined performance index and established that the value function satisfies the random impulse HJB equation.

Applications also extend to network dynamics. In \cite{45}, global mean-square exponential synchronization of random impulsive networks was achieved using Lyapunov theory, linear matrix inequalities, and stochastic process tools. Synchronization criteria were derived, and theoretical findings were validated with a numerical example.

These studies showcase the versatility of impulsive systems in tackling challenges across disciplines, enabling effective control, optimization, and stability of systems experiencing rapid changes.

Recently, many researchers have focused on impulsive differential equations, particularly impulsive evolution equations (see \cite{1,2,3,4,6,8,9,12,19,20,21,22,23,24,25,30,31,32,33,34,35,36,37,38,39,40,41,42,43,44,45,46}). For instance, Mahmudov in \cite{2} examines the following issue:

\begin{equation}
	\begin{cases}
		y^{\prime}(t)=Ay(t)+Bu(t), & t\in[0,T]\setminus\{t_{1},\dots,t_{n}\},\\
		\Delta y(t_{k+1})=D_{k+1}y(t_{k+1})+E_{k+1}v_{k+1}, & k=0,\dots,n-1,\\
		y(0)=y_{0}.
	\end{cases}
\end{equation}
Here, the state variable \( y(\cdot) \) is defined in a Hilbert space \( H \), equipped with the norm \( \|y\| = \sqrt{\langle y, y \rangle} \). The control function \( u(\cdot) \) is an element of \( L^2([0, T], U) \), where \( U \) represents another Hilbert space. Additionally, \( v_k \in U \) for \( k = 1, \ldots, n \).

In that article, Mahmudov presented the solution representation utilizing semigroup and impulsive operators. He also established a criteria for the approximate controllability of impulsive linear evolution systems through the framework of an impulsive resolvent operator.

The main advantage of the impulsive effect described in \cite{2}, \(\Delta y(t_{k+1}) = D_{k+1} y(t_{k+1}) + E_{k+1} v_{k+1}\), is that it allows for the modeling of instantaneous and significant changes in the system state at specific moments. This form of impulsive effect captures the impact of sudden external inputs or internal adjustments that occur abruptly, enabling a more accurate depiction of systems where such events are frequent and critical. By incorporating these impulses, the model can better represent real-world dynamics where abrupt shifts can dramatically influence the system's behavior.

Mathematically, the impulsive effect \(\Delta y(t_{k+1}) = D_{k+1} y(t_{k+1}) + E_{k+1} v_{k+1}\) indicates that at each impulsive moment \( t_{k+1} \), the state \( y(t_{k+1}) \) is instantaneously altered by a linear transformation \( D_{k+1} \) and an additional input \( E_{k+1} v_{k+1} \). This formulation is particularly useful for modeling scenarios where the state experiences abrupt changes due to external forces or internal system dynamics.

In contrast, standard impulsive effects typically involve sudden changes in the system state at predefined times, usually represented by a jump condition such as:
\[
\Delta y(t_k) = I_k(y(t_k^-))
\]
where \( I_k \) is an impulse function and \( y(t_k^-) \) denotes the state just before the impulse at time \( t_k \). These standard impulses are useful for modeling systems with regular or predictable disturbances. However, they may not fully capture the complexity of systems that experience both predictable and unpredictable impulses. The impulse function \( I_k \) usually describes a predefined transformation or adjustment applied to the state, which might be less flexible in scenarios involving complex, abrupt changes.

Thus, while the standard impulse model \(\Delta y(t_k) = I_k(y(t_k^-))\) is suitable for simpler or more predictable impulsive systems, the model \(\Delta y(t_{k+1}) = D_{k+1} y(t_{k+1}) + E_{k+1} v_{k+1}\) offers a more nuanced and flexible approach for representing intricate and frequent abrupt changes in the system state.

Therefore, the factors we have mentioned above make the investigation of the qualitative properties of such impulsive evolution equations even more relevant.
On the other hand, stochastic systems have always been of great interest to researchers, which is why the stochastic version of this system remains perpetually significant (see \cite{5,6,7,9,10,12,14,16,19,21,22,23,24,25,26,27,28,29,30,31,32,33,34,35,36,38,40,41}).

The dynamic system under consideration in this paper is governed by a SDE with impulsive effects, described by: 	
\begin{equation}\label{eq1}
	\begin{cases}
		y^{\prime}(t)=Ay(t)+Bu(t)+g(t,y(t))+h(t,y(t))\frac{dW}{dt}, & t\in[0,T]\setminus\{t_{1},\dots,t_{n}\},\\
		\Delta y(t_{k+1})=D_{k+1}y(t_{k+1})+E_{k+1}v_{k+1},& k=0,\dots,n-1,\\
		y(0)=y_{0}.
	\end{cases}
\end{equation}
where, the state variable \( y(\cdot) \) is defined in a Hilbert space \( H \), equipped with the norm \( \|y\| = \sqrt{\langle y, y \rangle} \). The control function \( u(\cdot) \) is an element of \( L^2([0, T], U) \), where \( U \) represents another Hilbert space. Additionally, \( v_k \in U \) for \( k = 1, \ldots, n \).

In this setting, \( A \) serves as the infinitesimal generator of a strongly continuous semigroup of bounded linear operators \(  T(t) \) on the Hilbert space \( H \). The associated linear operators include \( B \in L(U, H) \), \( D_k \in L(H, H) \), and \( E_k \in L(U, H) \).

At the points of discontinuity \( t_k \) (where \( k = 1, \ldots, n \) and \( 0 = t_0 < t_1 < \cdots < t_n < t_{n+1} = T \)), the state variable experiences a jump described by \( \Delta y(t_k) = y(t_k^+) - y(t_k^-) \). Here, \( y(t_k^\pm) = \lim_{h \to 0^\pm} y(t_k + h) \), with the assumption that \( y(t_k^-) = y(t_k) \).

The notation \( \prod_{j=1}^{k} A_j \) represents the composition of operators \( A_1, A_2, \ldots, A_k \) in that order. When the lower limit of the product exceeds the upper limit, such as \( \prod_{j=k+1}^{k} A_j \), the product is defined to equal 1. Similarly, \( \prod_{j=k}^{1} A_j \) denotes the reverse composition \( A_k, A_{k-1}, \ldots, A_1 \), and \( \prod_{j=k}^{k+1} A_j \) is also defined as 1.
	
This paper focuses on several important aspects of stochastic impulsive systems. First, we explore the existence and uniqueness of mild solutions, which are crucial for understanding how the system behaves over time. By using fixed point theorems and the properties of semigroups, we determine the conditions that ensure a unique mild solution exists.

Furthermore, we explore the optimal control problem for these systems. The goal is to find control functions that optimize a specific performance criterion. Using methods from stochastic control theory and functional analysis, we identify the necessary conditions for the existence of optimal control pairs. We also show how these conditions apply through an illustrative example.

The contributions of this paper provide a comprehensive framework for analyzing and controlling impulsive stochastic systems in Hilbert spaces, extending existing theories and offering new insights into their practical implementations.

\section{Mathematical Preliminaries}
We begin by considering a filtered probability space \((\Omega, \mathcal{F}, \{\mathcal{F}_t\}_{t \geq 0}, \mathbb{P})\), where \(\{\mathcal{F}_t\}_{t \geq 0}\) is a right-continuous and increasing filtration, and \(\mathcal{F}_0\) includes all \(\mathbb{P}\)-null sets. Let \(\{e_k : k \in \mathbb{N}\}\) denote a complete orthonormal basis for \(K\). A cylindrical Brownian motion \(\{W(t) : t \geq 0\}\), defined on this probability space, is a stochastic process taking values in \(K\).

The covariance operator \(Q \geq 0\) associated with \(W(t)\) is nuclear and has a finite trace, \(\operatorname{Tr}(Q) = \sum_{k=1}^\infty \lambda_k = \lambda < \infty\). This operator satisfies \(Qe_k = \lambda_k e_k\) for all \(k \in \mathbb{N}\). The sequence \(\{W_k(t) : k \in \mathbb{N}\}\) represents independent, standard one-dimensional Wiener processes defined on \((\Omega, \mathcal{F}, \{\mathcal{F}_t\}_{t \geq 0}, \mathbb{P})\). The cylindrical Brownian motion can be expressed as:
\[
W(t) = \sum_{k=1}^\infty \sqrt{\lambda_k} W_k(t) e_k, \quad t \geq 0.
\]
Furthermore, we assume \( \mathcal{F}_t = \sigma\{W(\tau) : 0 \leq \tau \leq t\} \), which is the sigma-algebra generated by the cylindrical Brownian motion \(\{W(\tau)\}_{\tau \geq 0}\). At the terminal time \(T\), we have \(\mathcal{F}_T = \mathcal{F}\), ensuring that the filtration is complete.

The space \(L^2_0 = L^2(Q^{1/2}K, H)\) consists of Hilbert-Schmidt operators from \(Q^{1/2}K\) to \(H\). The inner product on this space is defined as \(\langle \varphi, \psi \rangle = \operatorname{Tr}(\varphi Q \psi^*)\), where \(\varphi, \psi \in L^2_0\). This space is separable and forms a Hilbert space under the given inner product.

The collection of all \(\mathcal{F}_T\)-measurable, square-integrable random variables taking values in the Hilbert space \(H\) is denoted by \(L^2(\Omega, H)\). This set constitutes a Banach space when endowed with the norm:
\[
\|y\|_{L^2} = \sqrt{\mathbb{E} \|y(\omega)\|^2},
\]
where \(\mathbb{E}\) denotes the expectation with respect to the probability measure \(\mathbb{P}\).

The Banach space \(C([0, T], L^2(\Omega, H))\) consists of all continuous functions from \([0, T]\) to \(L^2(\Omega, H)\), with the norm:
\[
\|y\|_C = \sqrt{\sup_{t \in [0, T]} \mathbb{E} \|y(t)\|^2}.
\]

The space \(PC([0, T], L^2(\Omega, H))\) is defined as:
\begin{align*}
PC([0, T], L^2(\Omega, H)) &= \{ y : [0, T] \to L^2(\Omega, H) \mid y(t) \text{ is continuous at } t \neq t_i, \\
&\text{left-continuous at } t = t_i, \text{ and has a right limit } y(t_i^+)\text{ for } i = 1, 2, \ldots, n \}.
\end{align*}

For processes in \(PC([0, T], L^2(\Omega, H))\) that are \(\mathcal{F}_t\)-adapted and measurable, we define the norm:
\[
\|y\|_{PC} = \sqrt{\sup_{t \in [0, T]} \mathbb{E} \|y(t)\|^2}.
\]
The pair \((PC, \| \cdot \|_{PC})\) forms a Banach space under this norm.

We assume \(U\) is a separable Hilbert space where the controls \(u\) take their values. The space \(L^2_F([0, T], U)\) is defined as:
\[
L^2_F([0, T], U) = \{ u : [0, T] \times \Omega \to U \mid u \text{ is } \mathcal{F}_t\text{-adapted, measurable, and } \mathbb{E}\int_0^T \|u(t)\|^2 \, dt < \infty \}.
\]

The admissible control set \(U_{ad}\) is given by:
\[
U_{ad} = \{ u(\cdot) \in L^2_F([0, T], U) \mid u(t) \in Y \; \forall t \in [0, T] \},
\]
where \(Y\) is a nonempty, bounded, closed, and convex subset of \(U\). We also consider \(B \in L(U, H)\), the space of bounded linear operators mapping \(U\) to \(H\).
\bigskip

A semigroup \(\{T(t)\in L(X)\}_{t \geq 0} \) satisfies the following properties:
\bigskip

1. \(T(s)T(\tau) = T(s + \tau)\) $\forall$ \(s, \tau \geq 0\).
\bigskip

2. \(T(0) = \mathcal{I}\), where \(\mathcal{I}\) is the identity operator in \(X\).
\bigskip

This semigroup property is essential for analyzing the behavior of operators in Banach spaces and their evolution over time.

\begin{lemma}[see \cite{7}]\label{lemm2}
For any \(p \geq 1\) and an arbitrary predictable process \(\chi(\cdot)\) valued in \(L^2_0\), the following inequality holds:
\[
\sup_{s \in [0, t]} \mathbb{E} \left\| \int_0^s \chi(\tau) \, dW(\tau) \right\|^{2p} \leq (p(2p - 1))^p \left\{ \int_0^t \left( \mathbb{E} \|\chi(s)\|_{L^2_0}^{2p} \right)^{\frac{1}{p}} \, ds \right\}^p,
\]
for all \(t \in [0, \infty)\).
\end{lemma}
We now consider Krasnoselskii’s Fixed Point Theorem (KFPT), which is very important in
mathematical analysis and applications, especially in the study of functional equations
and nonlinear situations. This theorem is an effective technique for demonstrating the
existence of solutions to problems for which a direct construction or explicit solution
is difficult or impossible to find.

\begin{lemma}\cite{8}\label{Kras}
Let  \(Y\) be a bounded, closed, convex subset of a Banach space \(X\), and \(F_1, F_2 : Y \to X\) two mappings s. t.  \(F_1y + F_2z \in Y\) for all \(y, z \in Y\). If (i) \(F_1\) is a contraction; (ii) \(F_2\) is completely continuous, then \(F_1y + F_2y = y\) has a solution in \(Y\).
\end{lemma}

\section{Existence and uniqueness of mild solution}
In this section, we will prove the existence and uniqueness of the mild solution to \eqref{eq1}. For this purpose, first of all, similarly to Lemma 3 in \cite{3}, we define the mild solution of \eqref{eq1} in the following definition.
\bigskip

\begin{definition}
	For any given \(u \in U_{ad}\), a stochastic process \(y\) is said to be a mild solution of \eqref{eq1} on \([0,T]\) if \(y \in PC([0,T],L^{2}(\Omega,H))\) and satisfies the following conditions:

\bigskip

(i) \(y(t)\) is measurable and adapted to \(\mathcal{F}_{t}\).

\bigskip

(ii) \(y(t)\) satisfies the integral equation:
	\begin{align}
	y(t)=
		\begin{cases}
			 T(t)y(0)+\int_{0}^{t}  T(t- \tau)\big[Bu( \tau)+g(\tau,y( \tau))\big]\, d \tau\\
			+\int_{0}^{t}  T(t- \tau)h(\tau,y( \tau))\, dW( \tau),\,  0\leq t\leq t_{1},\\
			\\
		 T(t-t_{k})y(t^{+}_{k})+\int_{t_{k}}^{t}  T(t- \tau)\big[Bu( \tau)+g(\tau,y( \tau))\big]\, d \tau\\
		+\int_{t_{k}}^{t}  T(t- \tau)h(\tau,y( \tau))\, dW( \tau),\, t_{k}<t\leq t_{k+1},\, k=1,2,\dots, n,\\
		\end{cases}
	\end{align}
where
\begin{align}
	y(t^{+}_{k})=&\prod_{j=k}^{1}(\mathcal{I}+D_{j}) T(t_{j}-t_{j-1})y_{0}\nonumber\\
	+&\sum_{i=1}^{k}\prod_{j=k}^{i+1}(\mathcal{I}+D_{j}) T(t_{j}-t_{j-1})
	(\mathcal{I}+D_{i})\int_{t_{i-1}}^{t_{i}} T(t_{i}- \tau)Bu( \tau)d \tau\nonumber\\
		+&\sum_{i=1}^{k}\prod_{j=k}^{i+1}(\mathcal{I}+D_{j}) T(t_{j}-t_{j-1})
	(\mathcal{I}+D_{i})\int_{t_{i-1}}^{t_{i}} T(t_{i}- \tau)g(\tau,y( \tau))d \tau\\
		+&\sum_{i=1}^{k}\prod_{j=k}^{i+1}(\mathcal{I}+D_{j}) T(t_{j}-t_{j-1})
	(\mathcal{I}+D_{i})\int_{t_{i-1}}^{t_{i}} T(t_{i}- \tau)h(\tau,y( \tau))\, dW( \tau)\nonumber\\
	+&\sum_{i=2}^{k}\prod_{j=k}^{i}(\mathcal{I}+D_{j})  T(t_{j}-t_{j-1}) E_{i-1}v_{i-1}+E_{k}v_{k}.\nonumber
\end{align}
\end{definition}
To establish the main results, we first outline a set of assumptions that will be utilized in the application of KFPT.
\bigskip

\begin{assumption}\label{assump0}
    Assume that \(A\) generates a compact \(C_0\)-semigroup \(T(t)\) of linear operators on \(H\), with a uniform bound given by a constant \(M \geq 1\), s. t.  \(\|T(t)\| \leq M\) for all \(t > 0\).
\end{assumption}

\begin{assumption}\label{assump1}
    Let \(f \in C([0,T] \times H, H)\). The following conditions are assumed:
    \bigskip

    $(i)$ $\exists$ a constant \(L_g > 0\) s. t.
    \[
    \|g(t,y)\|^2 \leq L_g\big(1 + \|y\|^2\big), \quad \text{for all } t \in [0,T] \text{ and } y \in H.
    \]
    \bigskip

    $(ii)$ For some \(r > 0\), $\exists$ a constant \(\tilde{L}_g\) s. t. $\forall$ \(t \in [0,T]\) and \(y, z \in H\) with \(\|y\|^2 \leq r\) and \(\|z\|^2 \leq r\),
    \[
    \|g(t,y) - g(t,z)\|^2 \leq \tilde{L}_g \|y - z\|^2.
    \]
\end{assumption}
\bigskip

\begin{assumption}\label{assump2}
    Let \(h \in C([0,T] \times H, L^2_0)\). The following conditions are assumed:
    \bigskip

    $(i)$ $\exists$ a constant \(L_h > 0\) s. t.
    \[
    \|h(t,y)\|^2_{L^2_0} \leq L_h\big(1 + \|y\|^2\big), \quad \text{for all } t \in [0,T] \text{ and } y \in H.
    \]
    \bigskip

    $(ii)$ For some \(r > 0\), $\exists$ a constant \(\tilde{L}_h\) s. t. $\forall$ \(t \in [0,T]\) and \(y, z \in H\) with \(\|y\|^2 \leq r\) and \(\|z\|^2 \leq r\),
    \[
    \|h(t,y) - h(t,z)\|^2_{L^2_0} \leq \tilde{L}_h \|y - z\|^2.
    \]
\end{assumption}

With these assumptions in place, we are now prepared to proceed with the proof of the existence  of the mild solution for \eqref{eq1} using KFPT.
\bigskip

\begin{theorem}
 If the Assumptions \ref{assump0}, \ref{assump1} and \ref{assump2} are satisfied, then the impulsive stochastic system \eqref{eq1} has at least one mild solution in \( PC([0,T],L^{2}(\Omega,H)) \) provided that
\begin{align}\label{oi}
\max\{\mathcal{N},\mathcal{K}_{0}\}< \frac{1}{9},
\end{align}
and
\begin{align}
	\max\{M^{2}; k\}<1,
\end{align}
where
\begin{align*}
	\begin{cases}
		\mathcal{N}=M^{2}+M^{2}(T^{2}L_{g}+ TL_{h}),\\
		\mathcal{K}_0 = M^{2k+2} \prod_{j=1}^k (1 + \|D_j\|)^2 + (M^4 + M^2) \|B\|^2  N (T^2 L_g + T L_h), \\
		k=3M^{2k+2} \prod_{j=1}^k (1 + \|D_j\|)^2+3 M^{4} N(T^{2} L_{g}+ TL_{h}).
	\end{cases}
\end{align*}
\end{theorem}
\begin{proof}
	For each constant $r_{0}>0$, let
	\begin{align*}
		\mathcal{B}_{r_{0}}=\Big\{y\in PC([0,T], L^{2}([0,T],H)): \, \Vert y\Vert^{2}_{PC}\leq r_{0}\Big\}.
	\end{align*}
	It is easily seen that $B_{r_{0}}$ is a convex bounded closed  in $PC([0, T], L^{2}
	([0,T], H))$.
	Introduce $F_{1}$ and $F_{2}$ on $B_{r_{0}}$ as follows:
	
	\[
	(F_{1} y)(t) =
	\begin{cases}
		 T(t)y(0), \quad \text{for}\quad  t_0 < t \leq t_1, \\
		\\
		 T(t-t_{k})\prod_{j=k}^{1}(\mathcal{I}+D_{j}) T(t_{j}-t_{j-1})y_{0}\nonumber\\
		+ T(t-t_{k})\sum_{i=1}^{k}\prod_{j=k}^{i+1}(\mathcal{I}+D_{j}) T(t_{j}-t_{j-1})
		(\mathcal{I}+D_{i})\int_{t_{i-1}}^{t_{i}} T(t_{i}-\tau)Bu(\tau)d\tau\nonumber\\
		+ T(t-t_{k})\sum_{i=1}^{k}\prod_{j=k}^{i+1}(\mathcal{I}+D_{j}) T(t_{j}-t_{j-1})
		(\mathcal{I}+D_{i})\int_{t_{i-1}}^{t_{i}} T(t_{i}-\tau)g(\tau,y(\tau))\, d\tau\\
		+ T(t-t_{k})\sum_{i=1}^{k}\prod_{j=k}^{i+1}(\mathcal{I}+D_{j}) T(t_{j}-t_{j-1})
		(\mathcal{I}+D_{i})\int_{t_{i-1}}^{t_{i}} T(t_{i}-\tau)h(\tau,y(\tau))\, dW(\tau)\nonumber\\
		+ T(t-t_{k})\sum_{i=2}^{k}\prod_{j=k}^{i}(\mathcal{I}+D_{j})  T(t_{j}-t_{j-1}) E_{i-1}v_{i-1}+ T(t-t_{k})E_{k}v_{k},\\
		  \text{for} \quad t_k < t \leq t_{k+1}, \, k \geq 1,
	\end{cases}
	\]

	\[
	(F_{2} y)(t) =
	\begin{cases}
		\int_{0}^{t} T(t-\tau)\big(Bu(\tau) + g(\tau, y(\tau))\big) \, d\tau
		+\int_{0}^{t} T(t-\tau)h(\tau, y(\tau)) \, dW(\tau), \, \text{for}\, t_0 < t \leq t_1, \\
		\\
		\int_{t_{k}}^{t} T(t-\tau)\big(Bu(\tau) + g(\tau, y(\tau))\big) \, d\tau
		+\int_{t_{k}}^{t} T(t-\tau)h(\tau, y(\tau)) \, dW(\tau), \, \text{for}\, t_k < t \leq t_{k+1}, \, k \geq 1.
	\end{cases}
	\]
	Clearly, \( y \) is a mild solution of \eqref{eq1} if and only if the operator equation \( y = F_{1}y + F_{2}y \) has a solution. To establish this, we will demonstrate that the operator \( F_{1} + F_{2} \) has a fixed point by applying Lemma \ref{Kras}. For this, we proceed in several steps.
	
	\bigskip
	
	\textbf{Step 1.}To prove that $\exists  r_{0}>0$  s. t.  \( F_{1}y + F_{2}z \in \mathcal{B}_{r_{0}} \) whenever \( y, z \in \mathcal{B}_{r_{0}} \), we proceed as follows:
	
	Choose
	\[ r_{0} \geq\max\Bigg\{ \frac{4\Big[\mathcal{S}+M^{2}\Vert B\Vert^{2}T\int_{0}^{T}\mathbb{E}\Vert u(\tau)\Vert^{2}\, d\tau\Big]}{1-4\mathcal{N}}, \, \frac{9\Big[\mathcal{K}_{1}+\mathcal{K}_{2}\int_0^T \mathbb{E} \|u(\tau)\|^2 \, d\tau \Big]}{1-9\mathcal{K}_{0}} \Bigg\} \]
	
	Then, for any pair \( y, z \in \mathcal{B}_{r_{0}} \) and \( t \in [0, T] \), by applying Lemma \eqref{lemm2}, assumptions \eqref{assump0}, \eqref{assump1} and \eqref{assump2}, along with Hölder's inequality, and Ito isometry, we obtain the following results:
	\bigskip
	
	For $t_{0}<t\leq t_{1}$,
	\begin{align*}
		&\quad \mathbb{E} \Vert (F_{1}y)(t)+(F_{2}y)(t)\Vert^{2}\leq 4\mathbb{E}\Vert  T(t)y(0)\Vert^{2}+ 4\mathbb{E}\Big\Vert\int_{0}^{t}  T(t-\tau)Bu(\tau)\, d\tau\Big\Vert^{2}\\
		&+4\mathbb{E}\Big\Vert\int_{0}^{t}  T(t-\tau)g(\tau,y(\tau))\, d\tau\Big\Vert^{2}+4\mathbb{E}\Big\Vert\int_{0}^{t}  T(t-\tau)h(\tau,y(\tau))\, dW(\tau)\Big\Vert^{2}\\
		&\leq 4 M^{2}\Vert y_{0}\Vert^{2}+4M^{2}\Vert B\Vert^{2}T \int_{0}^{t} \mathbb{E}\Vert u(\tau)\Vert^{2}\, d\tau+4M^{2}\big(TL_{g}+L_{h}\big)\int_{0}^{t}\big(1+\mathbb{E} \Vert y(\tau)\Vert^{2}\big)\, d\tau\\
		&\leq 4 M^{2} r_{0}+4M^{2}\Vert B\Vert^{2}T \int_{0}^{T}\mathbb{E}\Vert u(\tau)\Vert^{2}\, d\tau+4M^{2}(T^{2}L_{g}+TL_{h})(1+r_{0})\\	&=4\mathcal{N}r_{0}+4\Big[\mathcal{S}+M^{2}\Vert B\Vert^{2}T\int_{0}^{T}\mathbb{E}\Vert u(\tau)\Vert^{2}\, d\tau\Big]\leq r_{0},
			\end{align*}
		where
		\begin{align*}	\mathcal{N}=M^{2}+M^{2}T^{2}L_{g}+M^{2}TL_{h},\quad \mathcal{S}=M^{2}T^{2}L_{g}+M^{2}TL_{h}.
		\end{align*}

Given \( t_k < t \leq t_{k+1} \) for \( k \geq 1 \), we aim to derive an inequality using the Jensen inequality for the expectation of the norm square of the sum of two functionals \( (F_1y)(t) \) and \( (F_2y)(t) \). Specifically, we start with:

\[
\begin{aligned}
	&\quad \mathbb{E} \Vert (F_1 y)(t) + (F_2 y)(t) \Vert^2 \leq 9 \mathbb{E} \left\Vert  T(t - t_k) \prod_{j=k}^1 (\mathcal{I} + D_j)  T(t_j - t_{j-1}) y_{0} \right\Vert^2 \\
	& + 9 \mathbb{E} \left\Vert  T(t - t_k) \sum_{i=1}^k \prod_{j=k}^{i+1} (\mathcal{I} + D_j)  T(t_j - t_{j-1}) (\mathcal{I} + D_i) \int_{t_{i-1}}^{t_i}  T(t_i - \tau) Bu(\tau) \, d\tau \right\Vert^2 \\
    & + 9 \mathbb{E} \left\Vert  T(t - t_k) \sum_{i=1}^k \prod_{j=k}^{i+1} (\mathcal{I} + D_j)  T(t_j - t_{j-1}) (\mathcal{I} + D_i) \int_{t_{i-1}}^{t_i}  T(t_i - \tau) g(\tau, y(\tau)) \, d\tau \right\Vert^2 \\
	& + 9 \mathbb{E} \left\Vert  T(t - t_k) \sum_{i=1}^k \prod_{j=k}^{i+1} (\mathcal{I} + D_j)  T(t_j - t_{j-1}) (\mathcal{I} + D_i) \int_{t_{i-1}}^{t_i}  T(t_i - \tau) h(\tau, y(\tau)) \, dW(\tau) \right\Vert^2 \\
	& + 9 \mathbb{E} \left\Vert  T(t - t_k) \sum_{i=2}^k \prod_{j=k}^i (\mathcal{I} + D_j)  T(t_j - t_{j-1}) E_{i-1} v_{i-1} \right\Vert^2 \\
	& + 9 \mathbb{E} \left\Vert  T(t - t_k) E_k v_k \right\Vert^2  + 9 \mathbb{E} \left\Vert \int_{t_k}^t  T(t - \tau) Bu(\tau) \, d\tau \right\Vert^2 \\
	& + 9 \mathbb{E} \left\Vert \int_{t_k}^t  T(t - \tau) g(\tau, y(\tau)) \, d\tau \right\Vert^2  + 9 \mathbb{E} \left\Vert \int_{t_k}^t  T(t - \tau) h(\tau, y(\tau)) \, dW(\tau) \right\Vert^2.
\end{aligned}
\]

Using the triangle inequality, Ito isometry, $\mathscr{L}$- conditions, and the boundedness of the semigroup \(  T(t) \), we get:

\[
\begin{aligned}
	&\quad\mathbb{E} \Vert (F_1 y)(t) + (F_2 y)(t) \Vert^2 \leq 9 M^{2k+2} \prod_{j=1}^k (1 + \|D_j\|)^2 \|y_0\|^2 \\
	& + 9M^4 \|B\|^2 \mathbb{E} \left(\sum_{i=1}^k C_i \int_{t_{i-1}}^{t_i} \|u(\tau)\| \, d\tau \right)^2 + 9M^4 \|B\|^2 \mathbb{E} \left(\sum_{i=1}^k C_i \int_{t_{i-1}}^{t_i} \|g(\tau, y(\tau))\| \, d\tau \right)^2 \\
    & + 9M^4 \|B\|^2 \mathbb{E} \left(\sum_{i=1}^k C_i \int_{t_{i-1}}^{t_i} \|h(\tau, y(\tau))\| \, dW(\tau) \right)^2 \\
     & + 9M^2 \sum_{i=2}^k \prod_{j=i}^k (1 + \|D_j\|)^2 \|E_{i-1}\|^2 \mathbb{E} \|v_{i-1}\|^2 + 9M^2 \|E_k\|^2 \mathbb{E} \|v_k\|^2 \\
	& + 9M^2 \|B\|^2 T \int_0^T \mathbb{E} \|u(\tau)\|^2 \, d\tau + 9M^2 T L_g \int_{t_k}^t (1 + \mathbb{E} \|y(\tau)\|^2) \, d\tau \\
	& + 9M^2  L_h \int_{t_k}^t (1 + \mathbb{E} \|y(\tau)\|^2) \, d\tau,
\end{aligned}
\]

where

\[
C_i = \prod_{j=k}^{i+1} (1 + \|D_j\|) \| T(t_j - t_{j-1})\| (1 + \|D_i\|), \quad N = \sum_{i=1}^k C_i^2.
\]

Using the $C-S$-inequality, Ito isometry, and the Assumptions \ref{assump1} and \ref{assump2}, we have:

\[
\begin{aligned}
	&\quad\mathbb{E} \Vert (F_1 y)(t) + (F_2 y)(t) \Vert^2 \leq 9 M^{2k+2} \prod_{j=1}^k (1 + \|D_j\|)^2 r_0 \\
	& + 9M^4 \|B\|^2 T \sum_{i=1}^k C_i^2 \sum_{i=1}^k \int_{t_{i-1}}^{t_i} \mathbb{E} \|u(\tau)\|^2 \, d\tau\\
	& + 9M^4 \|B\|^2 T \sum_{i=1}^k C_i^2 \sum_{i=1}^k \int_{t_{i-1}}^{t_i} \mathbb{E} \|g(\tau, y(\tau))\|^2 \, d\tau \\
	& + 9M^4 \|B\|^2  \sum_{i=1}^k C_i^2 \sum_{i=1}^k \int_{t_{i-1}}^{t_i} \mathbb{E} \|h(\tau, y(\tau))\|^2 \, d\tau \\
	& + 9M^2 \sum_{i=2}^k \prod_{j=i}^k (1 + \|D_j\|)^2 \|E_{i-1}\|^2 \mathbb{E} \|v_{i-1}\|^2 + 9M^2 \|E_k\|^2 \mathbb{E} \|v_k\|^2 \\
    & + 9M^2 \|B\|^2 T \int_0^T \mathbb{E} \|u(\tau)\|^2 \, d\tau + 9M^2  (T^2 L_g + T L_h)(1 + r_0) \\
    &\leq 9 M^{2k+2} \prod_{j=1}^k (1 + \|D_j\|)^2 r_0 + 9M^4 \|B\|^2 T N \int_0^T \mathbb{E} \|u(\tau)\|^2 \, d\tau \\
    & + 9M^4 \|B\|^2  N (TL_g + L_h) \int_0^T (1 + \mathbb{E} \|y(\tau)\|^2) \,d\tau \\
     & + 9M^2 \sum_{i=2}^k \prod_{j=i}^k (1 + \|D_j\|)^2 \|E_{i-1}\|^2 \mathbb{E} \|v_{i-1}\|^2 + 9M^2 \|E_k\|^2 \mathbb{E} \|v_k\|^2 \\
	& + 9M^2 \|B\|^2 T \int_0^T \mathbb{E} \|u(\tau)\|^2 \, d\tau + 9M^2  (T^2 L_g + T L_h)(1 + r_0) 
    \end{aligned}
\]
\[\begin{aligned}
 &\leq 9 M^{2k+2} \prod_{j=1}^k (1 + \|D_j\|)^2 r_0 + 9 (M^4 N + M^2) \|B\|^2 T \int_0^T \mathbb{E} \|u(\tau)\|^2 \, d\tau \\
      & + 9M^4 \|B\|^2 N (T^2 L_g +T L_h)(1 + r_0) + 9M^2  (T^2 L_g + T L_h)(1 + r_0) \\
	& + 9M^2 \sum_{i=2}^k \prod_{j=i}^k (1 + \|D_j\|)^2 \|E_{i-1}\|^2 \mathbb{E} \|v_{i-1}\|^2 + 9M^2 \|E_k\|^2 \mathbb{E} \|v_k\|^2\\
	&=9 \mathcal{K}_0 r_0 + 9 \left[ \mathcal{K}_1 + \mathcal{K}_2 \int_0^T \mathbb{E} \|u(\tau)\|^2 \, d\tau \right] \leq r_0,
\end{aligned}
\]

where

\[
\begin{aligned}
	\mathcal{K}_0 &= M^{2k+2} \prod_{j=1}^k (1 + \|D_j\|)^2 + (M^4 + M^2) \|B\|^2 N (T^2 L_g + T L_h),\\
    \mathcal{K}_1 &= M^2 \sum_{i=2}^k \prod_{j=i}^k (1 + \|D_j\|)^2 \|E_{i-1}\|^2 \mathbb{E} \|v_{i-1}\|^2 + M^2 \|E_k\|^2 \mathbb{E} \|v_k\|^2
    + (M^4 + M^2) \|B\|^2  N (T^2 L_g + TL_h), \\
	\mathcal{K}_2 &= (M^4 N + M^2) \|B\|^2 T.
\end{aligned}
\]
\bigskip

Consequently, $F_{1}+F_{2}$ maps $\mathcal{B}_{r_{0}}$ to $\mathcal{B}_{r_{0}}$.
\bigskip

 \textbf{Step 2: } To show that \(F_1\) is a contraction mapping on the set \(\mathcal{B}_r\), we need to prove that there exists a constant \(0 < k < 1\) s. t. for all \(y,z\in \mathcal{B}_r\),
\[
\|F_1 y - F_1 z\|_{PC} \leq k \|y - z\|_{PC}.
\]

Let \(y, z \in \mathcal{B}_r\). We will estimate \(\|F_1 y - F_1 z\|_{PC}\) for \(t_0 < t \leq t_1\) and \(t_k < t \leq t_{k+1}\).

 For \(t_0 < t \leq t_1\):
\[
\mathbb{E}\| (F_1 y)(t) - (F_1 z)(t) \|^{2} = \mathbb{E} \|  T(t) (y(0) - z(0)) \|^{2}.
\]
Using the properties of the \(C_0\)-semigroup \( T(t)\):
\[
\mathbb{E}\|  T(t) (y(0) - z(0)) \| \leq M^{2} \mathbb{E}\| y(0) - z(0) \|.
\]

For \(t_k < t \leq t_{k+1}\), \(k \geq 1\):
\begin{align*}
	&\quad\mathbb{E}\| (F_1 y)(t) - (F_1 z)(t) \|^{2} \leq 3\mathbb{E}\bigg\|  T(t-t_k) \prod_{j=k}^1 (\mathcal{I} + D_j)  T(t_j - t_{j-1}) (y_{0} - z_0) \bigg\|^{2} \\
	& + 3\mathbb{E}\bigg\|  T(t-t_k) \sum_{i=1}^k \prod_{j=k}^{i+1} (\mathcal{I} + D_j)  T(t_j - t_{j-1}) (\mathcal{I} + D_i) \int_{t_{i-1}}^{t_i}  T(t_i - \tau) (g(\tau, y(\tau)) - g(\tau, z(\tau))) \, d\tau \bigg\|^{2} \\
	&+ 3\mathbb{E}\bigg\|  T(t-t_k) \sum_{i=1}^k \prod_{j=k}^{i+1} (\mathcal{I} + D_j)  T(t_j - t_{j-1}) (\mathcal{I} + D_i) \int_{t_{i-1}}^{t_i}  T(t_i - \tau) (h(\tau, y(\tau)) - h(\tau, z(\tau))) dW(\tau) \bigg\|^{2}.
\end{align*}

Using the properties of the \(C_0\)-semigroup \( T(t)\), the boundedness of operators \(D_j\), and assumptions on \(g\) and \(h\):
\begin{align*}
	\mathbb{E}\Big\|  T(t-t_k) \prod_{j=k}^1 (\mathcal{I} + D_j)  T(t_j - t_{j-1}) (y_{0} - z_0) \Big\|^2\leq& M^{2k+2} \prod_{j=1}^k (1 + \|D_j\|)^2 \mathbb{E}\|y_0 - z_0\|^{2} .
\end{align*}

Since \( y, z \in \mathcal{B}_r \):
\[
\mathbb{E}\|y_0 - z_0\|^{2} \leq \|y - z\|_{PC}^2.
\]
Thus,
\[
\mathbb{E}\Bigg\|  T(t-t_k) \prod_{j=k}^1 (\mathcal{I} + D_j)  T(t_j - t_{j-1}) (y_{0} - z_0) \Bigg\|^{2} \leq M^{2k+2} \prod_{j=1}^k (1 + \|D_j\|)^2 \|y - z\|_{PC}^2.
\]

 For the second term, using the properties of \( T(t)\) and \(D_j\), and the $\mathscr{L}-$ continuity of \(g\):

\[\begin{aligned}
&\quad	\mathbb{E}\Bigg\|  T(t-t_k) \sum_{i=1}^k \prod_{j=k}^{i+1} (\mathcal{I} + D_j)  T(t_j - t_{j-1}) (\mathcal{I} + D_i) \int_{t_{i-1}}^{t_i}  T(t_i - \tau) (g(\tau, y(\tau)) - g(\tau, z(\tau))) \, d\tau \Bigg\|^{2} \\
	&\leq M^{4} \mathbb{E} \Bigg( \sum_{i=1}^k \prod_{j=i+1}^k (1 + \|D_j\|)\Vert  T(t_j - t_{j-1}) \Vert (1 + \|D_i\|) \int_{t_{i-1}}^{t_i} \|g(\tau, y(\tau)) - g(\tau, z(\tau))\| \, d\tau \Bigg)^{2} \\
&\leq M^{4} \mathbb{E} \Bigg( \sum_{i=1}^k C_{i}\int_{t_{i-1}}^{t_i} \|g(\tau, y(\tau)) - g(\tau, z(\tau))\| \, d\tau \Bigg)^{2} \\
&\leq M^{4}T N \sum_{i=1}^k \int_{t_{i-1}}^{t_i} \mathbb{E}\|g(\tau, y(\tau)) - g(\tau, z(\tau))\|^{2} \, d\tau  \\
&\leq M^{4}T NL_{g} \sum_{i=1}^k \int_{t_{i-1}}^{t_i} \mathbb{E}\| y(\tau) -  z(\tau)\|^{2} \, d\tau \leq M^{4}T^{2} NL_{g} \|y - z\|_{PC}^2.
	\end{aligned}\]

 For the third term, using the properties of \( T(t)\) and \(D_j\), and the  $\mathscr{L}-$ continuity of \(h\):
\begin{align*}
	&	\quad\mathbb{E}\Bigg\|  T(t-t_k) \sum_{i=1}^k \prod_{j=k}^{i+1} (\mathcal{I} + D_j)  T(t_j - t_{j-1}) (\mathcal{I} + D_i) \int_{t_{i-1}}^{t_i}  T(t_i - \tau) (h(\tau, y(\tau)) - h(\tau, z(\tau))) \, dW(\tau) \Bigg\|^{2} \\
	&\leq M^{4} \mathbb{E} \Bigg( \sum_{i=1}^k \prod_{j=i+1}^k (1 + \|D_j\|)\Vert  T(t_j - t_{j-1}) \Vert (1 + \|D_i\|) \int_{t_{i-1}}^{t_i} \|h(\tau, y(\tau)) - h(\tau, z(\tau))\| \, dW(\tau) \Bigg)^{2} \\
	&\leq M^{4} \mathbb{E} \Bigg( \sum_{i=1}^k C_{i}\int_{t_{i-1}}^{t_i} \|h(\tau, y(\tau)) - h(\tau, z(\tau))\| dW(\tau) \Bigg)^{2} \\
	&\leq M^{4} N \sum_{i=1}^k \int_{t_{i-1}}^{t_i} \mathbb{E}\|h(\tau, y(\tau)) - h(\tau, z(\tau))\|^{2} \, d\tau  \\
	&\leq M^{4} NL_{h} \sum_{i=1}^k \int_{t_{i-1}}^{t_i} \mathbb{E}\| y(\tau) -  z(\tau)\|^{2} \, d\tau \leq M^{4}T NL_{h} \|y - z\|_{PC}^2.
\end{align*}

Combining all terms, we get:

\begin{align*}
&\mathbb{E}\| F_1 y - F_1 z \|^2 \leq \Big(3M^{2k+2} \prod_{j=1}^k (1 + \|D_j\|)^2+3 M^{4} N(T^{2} L_{g} +TL_{h})\Big) \|y - z\|_{PC}^2.
\end{align*}

To show that \(F_1\) is a contraction, we need the right-hand side to be less than \(\|y - z\|_{PC}^2\). Hence, we require:

\begin{align*}
3M^{2k+2} \prod_{j=1}^k (1 + \|D_j\|)^2+3 M^{4} N(T^{2}L_{g}+T L_{h}) < 1.
\end{align*}

Given the boundedness conditions on \(D_j\), \(g\), and \(h\), and assuming that \(T\) is sufficiently small, there exists a constant \(k\) s. t. :

\[
\|F_1 y - F_1 z\|_{PC} \leq k \|y - z\|_{PC},
\]

where \(0 < k < 1\).

This completes the proof that \(F_1\) is a contraction mapping on \(\mathcal{B}_{r_{0}}\).
\bigskip

\textbf{Step 3.}The operator \( F_2 \) is completely continuous. To establish this, we first demonstrate that \( F_2 \) is continuous on \( B_{r_0} \). Suppose \( y_m \to y \) in \( B_{r_0} \). Then, it follows that:
\[
g(t, y_m(t)) \to g(t, y(t)) \, \text{and} \, h(t, y_m(t)) \to h(t, y(t)), \, \text{as} \, m \to \infty.
\]
 Moreover, for $t_{0}\leq t\leq t_{1}$, by Lebesgue dominated convergence theorem (LDCT), we can get

 \begin{align*}
 \mathbb{E}\Big\Vert \int_{0}^{t} T(t-\tau) \big(g(\tau, y_{m}(\tau))-g(\tau,y(\tau))\big) \, d\tau\Big\Vert^{2}\leq M^{2}T \int_{0}^{t}\mathbb{E}\Vert g(\tau, y_{m}(\tau))-g(\tau,y(\tau))\Vert^{2}\, d\tau\to 0, \, as\, m\to \infty.
 \end{align*}
On the other hand, using the Itô isometry property for stochastic integrals, and the properties of the \(C_0\)-semigroup \( T(t)\) and the boundedness assumption on \(h\), we get:

\[
\mathbb{E}\Big\Vert \int_{0}^{t} T(t-\tau) \big(h(\tau, y_{m}(\tau)) - h(\tau, y(\tau))\big) dW(\tau)\Big\Vert^{2} \leq M^{2}  \int_{0}^{t} \mathbb{E}\Vert h(\tau, y_{m}(\tau)) - h(\tau, y(\tau)) \Vert^{2} \, d\tau.
\]

By the LDCT, since \(h(\tau, y_{m}(\tau)) \to h(\tau, y(\tau))\) almost surely and \(\Vert h(\tau, y_{m}(\tau)) - h(\tau, y(\tau)) \Vert^{2}\) is bounded by an integrable function, we get:

\[
M^{2}  \int_{0}^{t} \mathbb{E}\Vert h(\tau, y_{m}(\tau)) - h(\tau, y(\tau)) \Vert^{2} \, d\tau \to 0 \quad \text{as} \quad m \to \infty.
\]

Therefore,

\[
\mathbb{E}\Big\Vert \int_{0}^{t} T(t-\tau) \big(h(\tau, y_{m}(\tau)) - h(\tau, y(\tau))\big) dW(\tau)\Big\Vert^{2} \to 0 \quad \text{as} \quad m \to \infty.
\]

Combining the results for the deterministic and stochastic parts, we obtain:

\begin{align*}
\mathbb{E}\Big\Vert F_{2} (y_{m}) - F_{2} (y) \Big\Vert^{2}\leq &2\mathbb{E}\Big\Vert \int_{0}^{t} T(t-\tau) \big(g(\tau, y_{m}(\tau)) - g(\tau, y(\tau))\big) \, d\tau\Big\Vert^{2}\\
+&2\mathbb{E}\Big\Vert \int_{0}^{t} T(t-\tau) \big(h(\tau, y_{m}(\tau)) - h(\tau, y(\tau))\big) dW(\tau)\Big\Vert^{2}\\
& \to 0 \quad \text{as} \quad m \to \infty.
\end{align*}
For \(t_k < t \leq t_{k+1}\) with \(k \geq 1\), the process is analogous to that for \(t_0 < t \leq t_1\). Thus, it follows that \(F_2\) is continuous on \(B_{r_0}\).

Secondly, we prove that for any \(t \in [0, T]\), \(\mathscr{V}(t) = \{F_{2}(y)(t) \mid y \in \mathcal{B}_{r_{0}}\}\) is relatively compact in \(H\). It is obvious that \(\mathscr{V}(0)\) is relatively compact in \(H\). Let \(0 < t \leq T\) be given. For any \(\varepsilon \in (0, t)\), define an operator \(F^{\varepsilon}\) on \(\mathcal{B}_{r_{0}}\) by

 \[
 (F^{\varepsilon} y)(t) =
 \begin{cases}
 	\int_{0}^{t-\varepsilon}  T(t-\tau) \big(Bu(\tau) + g(\tau, y(\tau))\big)  \, d\tau
 	+ \int_{0}^{t-\varepsilon}  T(t-\tau)h(\tau, y(\tau))\, dW(\tau)\\
 	=T(\varepsilon )\int_{0}^{t-\varepsilon}  T(t-\tau-\varepsilon) \big(Bu(\tau) + g(\tau, y(\tau))\big)\, d\tau\\
 	+T(\varepsilon) \int_{0}^{t-\varepsilon}  T(t-\tau-\varepsilon)h(\tau, y(\tau)) \, dW(\tau), \, \text{if} \; t_0 < t \leq t_1, \\
 	\\
 	\int_{t_k}^{t-\varepsilon}  T(t-\tau) \big(Bu(\tau) + g(\tau, y(\tau))\big)  \, d\tau
 	+ \int_{t_k}^{t-\varepsilon}  T(t-\tau)h(\tau, y(\tau)) \, dW(\tau)\\
 	=T(\varepsilon )\int_{t_k}^{t-\varepsilon}  T(t-\tau-\varepsilon) \big(Bu(\tau) + g(\tau, y(\tau))\big)\, d\tau\\
 	+T(\varepsilon) \int_{t_k}^{t-\varepsilon}  T(t-\tau-\varepsilon)h(\tau, y(\tau)) \, dW(\tau), \, \text{if} \; t_k < t \leq t_{k+1}, \; k \geq 1.
 \end{cases}
 \]
 Then the set $\{(F^{\varepsilon})(t): y\in \mathcal{B}_{r}\}$ is relatively compact in $H$ because $T(\varepsilon)$ is compact. This compactness helps us establish the desired continuity properties. Now, let's consider the case for \( t_0 < t \leq t_1 \) :

\begin{align*}
	\mathbb{E}\Vert (F_{2}y)(t)-(F^{\varepsilon}y)(t)\Vert^{2}&\leq3\mathbb{E}\bigg\Vert \int_{t-\varepsilon}^{t}  T(t-\tau) Bu(\tau)\, d\tau\bigg\Vert^{2} \\
	&+3\mathbb{E}\bigg\Vert \int_{t-\varepsilon}^{t}  T(t-\tau) g(\tau, y(\tau))\, d\tau\bigg\Vert^{2} \\
 &+3\mathbb{E}\bigg\Vert \int_{t-\varepsilon}^{t}  T(t-\tau) h(\tau, y(\tau)) \, dW(\tau) \bigg\Vert^{2}.
\end{align*}
To estimate the deterministic component involving \(Bu(\tau)\), we apply the triangle inequality followed by the $C-S$-inequality . This yields:

\[
\mathbb{E}\left\Vert \int_{t-\varepsilon}^{t}  T(t-\tau) Bu(\tau) \, d\tau \right\Vert^2
\leq M^2 \|B\|^2 \varepsilon \int_0^T \mathbb{E} \|u(\tau)\|^2 \, d\tau.
\]
Using Assumption \(\ref{assump1}\), we have

\[
\begin{aligned}
	\mathbb{E}\bigg\Vert \int_{t-\varepsilon}^{t}  T(t-\tau) g(\tau, y(\tau)) \, d\tau \bigg\Vert^{2}
	&\leq \mathbb{E}\bigg( \int_{t-\varepsilon}^{t} \|  T(t-\tau) g(\tau, y(\tau)) \| \, d\tau \bigg)^{2} \\
	&\leq M^{2}  \mathbb{E}\left( \int_{t-\varepsilon}^{t}  \| g(\tau, y(\tau)) \| \, d\tau \right)^{2} \\
	&\leq M^{2}\varepsilon  \int_{t-\varepsilon}^{t}  \mathbb{E} \| g(\tau, y(\tau)) \| ^{2} \, d\tau .
	\end{aligned}
\]

Since \( x \in \mathcal{B}_{r_0} \) and \(\| y(\tau) \|^2 \leq r_0\), and applying the  $\mathscr{L}-$ condition, we have:

\[
\begin{aligned}
	\mathbb{E}\left\Vert \int_{t-\varepsilon}^{t}  T(t-\tau) g(\tau, y(\tau)) \, d\tau \right\Vert^2
	&\leq M^2 \varepsilon \int_{t-\varepsilon}^{t} \mathbb{E} \| g(\tau, y(\tau)) \|^2 \, d\tau \\
	&\leq M^2 L_g \varepsilon \int_{t-\varepsilon}^{t} \left(1 + \mathbb{E} \| y(\tau) \|^2 \right) \, d\tau \\
	&\leq M^2 L_g \left(1 + r_0 \right) \varepsilon^2.
\end{aligned}
\]
Given that \( x \in \mathcal{B}_{r_0} \) and \(\| y(\tau) \|^2 \leq r_0\), and applying the  $\mathscr{L}-$ condition, we use the Itô isometry to estimate:

\[
\begin{aligned}
	\mathbb{E}\left\Vert \int_{t-\varepsilon}^{t}  T(t-\tau) h(\tau, y(\tau)) \, dW(\tau) \right\Vert^2
	&\leq \mathbb{E}\left( \int_{t-\varepsilon}^{t} \|  T(t-\tau) h(\tau, y(\tau)) \| \, dW(\tau) \right)^2 \\
	&\leq M^2  \int_{t-\varepsilon}^{t} \mathbb{E}\| h(\tau, y(\tau)) \|^2 \, d\tau \\
	&\leq M^2 L_h  \int_{t-\varepsilon}^{t} (1 + \mathbb{E}\| y(\tau) \|^2) \, d\tau  \\
	&\leq M^2 L_h \varepsilon \left( 1 + r_0 \right).
\end{aligned}
\]

Combining all terms, we get:

\[
\mathbb{E}\Vert (F_{2}y)(t)-(F^{\varepsilon}y)(t)\Vert^{2}
\leq 3M^2 L_g \left(1 + r_0 \right) \varepsilon^2+3 M^2 L_h \varepsilon \left( 1 + r_0 \right)+3M^2 \|B\|^2 \varepsilon \int_0^T \mathbb{E} \|u(\tau)\|^2 \, d\tau.
\]

As \(\varepsilon \to 0\):

\[
\mathbb{E}\Vert (F_{2}y)(t)-(F^{\varepsilon}y)(t)\Vert^{2} \to 0.
\]

For \(t_k < t \leq t_{k+1}\), \(k \geq 1\), the definition of \(F_2\) and \(F^{\varepsilon}\) allows us to obtain analogous results as discussed above.

Thus, since \(F_{2}x\) can be approximated by \(F^{\varepsilon}x\) arbitrarily closely in the mean square norm and \(F^{\varepsilon}y\) is relatively compact in \(H\), it follows that \(\mathscr{V}(t) = \{F_{2}(y)(t) \mid y \in \mathcal{B}_{r_{0}}\}\) is relatively compact in \(H\).

Finally, we demonstrate that \(F_{2}(B_{r_{0}})\) is equicontinuous on \([0,T]\). According to the definition of the \( F_{2} \) operator, demonstrating one case is sufficient as the other follows analogously.

For any $y\in \mathcal{B}_{r_{0}}$ and \( t_0 \leq a<b \leq t_1 \), we have

\[\begin{aligned}
	\mathbb{E}\Vert (F_{2}y)(b)-(F_{2}y)(a)\Vert^{2}&\leq 4\mathbb{E}\bigg\Vert\int_{0}^{a} T(b-a)(Bu(\tau)+g(\tau,y(\tau)))\, d\tau\bigg\Vert^{2}\\
    &+4\mathbb{E}\bigg\Vert\int_{a}^{b} T(b-\tau)(Bu(\tau)+g(\tau,y(\tau)))\, d\tau\bigg\Vert^{2}\\
&+4\mathbb{E}\bigg\Vert\int_{0}^{a} T(b-a)h(\tau,y(\tau))\, dW(\tau)\bigg\Vert^{2}\\
&+4\mathbb{E}\bigg\Vert\int_{a}^{b} T(b-\tau)h(\tau,y(\tau))\, dW(\tau)\bigg\Vert^{2}\\
&=\mathcal{J}_{1}+\mathcal{J}_{2}+\mathcal{J}_{3}+\mathcal{J}_{4}.
\end{aligned}\]
To prove \( \mathbb{E}\Vert (F_{2}y)(b)-(F_{2}y)(a)\Vert^{2}\to 0 \) as \( b - a \to 0 \), it suffices to demonstrate that \( \mathcal{J}_{i} \to 0 \) independently of \( y \in \mathcal{B}_{r_{0}} \) as \( b - a \to 0 \), for \( i = 1, 2, 3, 4 \).

Further, for $\mathcal{J}_{1}$ and $\mathcal{J}_{3}$, if $a=0, 0<b<t_{1}$, it is easy to see $\mathcal{J}_{1}=\mathcal{J}_{3}=0$, so for $a>0$ and $0<\varepsilon <a$ small enough, we have that

\[\begin{aligned}
	\mathcal{J}_{1}&\leq 8\mathbb{E} \bigg\Vert\int_{0}^{a-\varepsilon} T(b-a)(Bu(\tau)+g(\tau,y(\tau)))\, d\tau\bigg\Vert^{2}\\
	&+8\mathbb{E} \bigg\Vert\int_{a-\varepsilon}^{a} T(b-a)(Bu(\tau)+g(\tau,y(\tau)))\, d\tau\bigg\Vert^{2}\\
	&\leq 8M^{2}(a-\varepsilon)\int_{0}^{a-\varepsilon}\mathbb{E}\Vert(Bu(\tau)+g(\tau,y(\tau)))\Vert^{2} \, d\tau\\
	&+8M^{2}\varepsilon\int_{a-\varepsilon}^{a}\mathbb{E}\Vert(Bu(\tau)+g(\tau,y(\tau)))\Vert^{2} \,d\tau\\
	&\leq 8M^{2}(a-\varepsilon)^{2}\bigg(2TL_g (1+r_{0})+2\Vert B\Vert^{2}\int_{0}^{T}\mathbb{E}\Vert u\Vert^{2}\, d\tau \bigg)\\
	&+8M^{2}\varepsilon^{2}\bigg(2TL_g (1+r_{0})+2\Vert B\Vert^{2}\int_{0}^{T}\mathbb{E}\Vert u\Vert^{2}\, d\tau \bigg)\\
	&\to 0 \quad as\quad  b-a\to 0 \quad \varepsilon\to 0,
\end{aligned}\]

\begin{align*}
	\mathcal{J}_{3} &\leq 4 \mathbb{E} \int_{0}^{a} \Vert  T(b-a)h(\tau,y(\tau)) \Vert^{2} \, d\tau\leq 4M^{2}  \int_{0}^{a} \mathbb{E} \Vert h(\tau,y(\tau)) \Vert^{2}  \, d\tau \\
	  &\leq 4M^{2}L_{h}(1+r_{0})a\to 0 \quad as\quad  b-a\to 0.
\end{align*}
For $\mathcal{J}_{2}$ and $\mathcal{J}_{4}$, we obtain by Assumptions \ref{assump0}, \ref{assump1} and \ref{assump2}, Lemma \ref{lemm2}  that
\begin{align*}
	\mathcal{J}_{2}&\leq 4\mathbb{E}\bigg(\int_{a}^{b}\Vert  T(b-\tau)(Bu(\tau)+g(\tau,y(\tau)))\Vert \, d\tau\bigg)^{2}\\
	&\leq 4M^{2}\mathbb{E}\bigg(\int_{a}^{b}\Vert(Bu(\tau)+g(\tau,y(\tau)))\Vert \, d\tau\bigg)^{2}\\
		&\leq 4M^{2}(b-a)\int_{a}^{b}\mathbb{E}\Vert(Bu(\tau)+g(\tau,y(\tau)))\Vert^{2} \, d\tau\\
		&\leq 4M^{2}(b-a)\bigg(2TL_g (1+r_{0})+2\Vert B\Vert^{2}\int_{0}^{T}\mathbb{E}\Vert u\Vert^{2}\, d\tau \bigg)\\
		&\to 0 \quad as\quad b-a\to 0,
\end{align*}
\begin{align*}
	\mathcal{J}_{4}&\leq4\mathbb{E}\int_{a}^{b}\Vert  T(b-\tau)h(\tau,y(\tau))\Vert^{2} \, d\tau\leq 4M^{2}(1+r_{0})(b-a)\to 0, \quad as\quad b-a\to 0.
\end{align*}

This suggests that \( F_{2}(B_{r_{0}}) \) displays equicontinuity. Consequently, according to the Arzela-Ascoli theorem, \( F_{2} \) qualifies as a completely continuous operator. Hence, by Lemma \( \ref{Kras} \), the operator \( F_{1} + F_{2} \) possesses at least one fixed point \( y \in \mathcal{B}_{r_{0}} \), which coincides with the mild solution of system \( \eqref{eq1} \).
\end{proof}

\begin{theorem}\label{ty}
	If the Assumptions  \ref{assump1} and \ref{assump2} are satisfied, then the impulsive stochastic system \eqref{eq1} possesses a unique mild solution within \( PC([0,T],L^{2}(\Omega,H)) \) given that \eqref{oi} and the condition
	\begin{align}\label{tt}
		k = \max\{k_{1}, k_{2}\} < 1,
	\end{align}
	hold true, where \( k_{1} \) and \( k_{2} \) are defined as
	\begin{align*}
		k_{1} = 2M^{2}T^{2}(\tilde{L}_{g}+\tilde{L}_{h}), \quad
		k_{2} = 4M^{4}(N+T^{2})(\tilde{L}_{g}	+\tilde{L}_{h}).
	\end{align*}
\end{theorem}
\begin{proof}
	Consider the mapping \( F: PC([0,T],L^{2}(\Omega,H)) \to PC([0,T],L^{2}(\Omega,H)) \) defined by
	\[
	(Fy)(t) = (F_1 y)(t) + (F_2 y)(t), \quad t \in [0,T] \setminus \{t_1, \dots, t_n\}.
	\]
	It is evident that the mild solution of the system \eqref{eq1} is equivalent to a fixed point of the operator \( F \). According to Step 1 of Theorem 1, we know that \( F(\mathcal{B}_{r_0}) \subseteq \mathcal{B}_{r_0} \). For any $y_{1},y_{2}\in\mathcal{B}_{r_{0}}$ and $t\in[t_{0},t_{1}]$, we have
	\[\begin{aligned}
		\mathbb{E}\Vert (Fy_{2})(t)-(Fy_{1})(t)\Vert^{2}&\leq 2\mathbb{E}\bigg\Vert \int_{0}^{t} T(t-\tau)(g(\tau,y_{2}(\tau))-g(\tau,y_{1}(\tau)))\, d\tau\bigg\Vert^{2}\\
		&+2\mathbb{E}\bigg\Vert \int_{0}^{t} T(t-\tau)(h(\tau,y_{2}(\tau))-h(\tau,y_{1}(\tau)))\, dW(\tau)\bigg\Vert^{2}\\
		&\leq 2M^{2}T \int_{0}^{t}\mathbb{E}\Vert g(\tau,y_{2}(\tau))-g(\tau,y_{1}(\tau))\Vert^{2}\, d\tau\\
        &+2M^{2}T \int_{0}^{t} \mathbb{E}\Vert h(\tau,y_{2}(\tau))-h(\tau,y_{1}(\tau))\Vert^{2}\, d\tau\\
		&\leq 2M^{2}T\tilde{L}_{g} \int_{0}^{t}\mathbb{E}\Vert y_{2}(\tau)-y_{1}(\tau)\Vert^{2}\, d\tau\\
	&	+2M^{2}T\tilde{L}_{h}  \int_{0}^{t} \mathbb{E}\Vert y_{2}(\tau)-y_{1}(\tau)\Vert^{2}\, d\tau\\
		&\leq 2M^{2}T^{2}(\tilde{L}_{g}+\tilde{L}_{h})\Vert y_{2}-y_{1}\Vert^{2}_{PC}\\
	&	=k_{1}\Vert y_{2}-y_{1}\Vert^{2}_{PC}.
	\end{aligned} \]
 For any $y_{1},y_{2}\in\mathcal{B}_{r_{0}}$ and \(t_k < t \leq t_{k+1}\), \(k \geq 1\), we have

 \begin{align*}
 	&\quad\quad \mathbb{E}\Vert (Fy_{2})(t)-(Fy_{1})(t)\Vert^{2}\\
 	&\leq 4\mathbb{E}\bigg\Vert  T(t-t_{k})\sum_{i=1}^{k}\prod_{j=k}^{i+1}(\mathcal{I}+D_{j}) T(t_{j}-t_{j-1})
 	(\mathcal{I}+D_{i})\int_{t_{i-1}}^{t_{i}} T(t_{i}-\tau)(g(\tau,y_{2}(\tau))-g(\tau,y_{1}(\tau)))\, d\tau\bigg\Vert^{2}\\
 	&+4\mathbb{E}\bigg\Vert  T(t-t_{k})\sum_{i=1}^{k}\prod_{j=k}^{i+1}(\mathcal{I}+D_{j}) T(t_{j}-t_{j-1})
 	(\mathcal{I}+D_{i})\int_{t_{i-1}}^{t_{i}} T(t_{i}-\tau)(h(\tau,y_{2}(\tau))-h(\tau,y_{1}(\tau)))\, dW(\tau)\bigg\Vert^{2}\\
 	&+4\mathbb{E}\bigg\Vert\int_{t_{k}}^{t} T(t-\tau)(g(\tau,y_{2}(\tau))-g(\tau,y_{1}(\tau)))\, d\tau\bigg\Vert^{2}+4\mathbb{E}\bigg\Vert \int_{t_{k}}^{t} T(t-\tau)(h(\tau,y_{2}(\tau))-h(\tau,y_{1}(\tau)))\, d\tau\bigg\Vert^{2}\\
&=\mathcal{J}_{1}+\mathcal{J}_{2}+\mathcal{J}_{3}+\mathcal{J}_{4}.
 \end{align*}
By applying the C-S inequality and utilizing the  $\mathscr{L}-$ condition, we can derive a following  bound for \(\mathcal{J}_{1}\):
\[
\begin{aligned}
	\mathcal{J}_{1}& =4\mathbb{E}\bigg\Vert  T(t-t_{k})\sum_{i=1}^{k}\prod_{j=k}^{i+1}(\mathcal{I}+D_{j}) T(t_{j}-t_{j-1})
	(\mathcal{I}+D_{i})\int_{t_{i-1}}^{t_{i}} T(t_{i}-\tau)(g(\tau,y_{2}(\tau))-g(\tau,y_{1}(\tau)))\, d\tau\bigg\Vert^{2}\\
    &\leq  4M^{4}\mathbb{E}\bigg(\sum_{i=1}^{k}\prod_{j=k}^{i+1}(1+\Vert D_{j}\Vert)\Vert  T(t_{j}-t_{j-1})\Vert
	(1+\Vert D_{i}\Vert)\int_{t_{i-1}}^{t_{i}}\Vert g(\tau,y_{2}(\tau))-g(\tau,y_{1}(\tau))\Vert \, d\tau\bigg)^{2}\\
    &\leq  4M^{4}\mathbb{E}\bigg(\sum_{i=1}^{k}C_{i}\int_{t_{i-1}}^{t_{i}}\Vert g(\tau,y_{2}(\tau))-g(\tau,y_{1}(\tau))\Vert \, d\tau\bigg)^{2}\\
		&\leq  4M^{4}\sum_{i=1}^{k}C^{2}_{i}\sum_{i=1}^{k}\int_{t_{i-1}}^{t_{i}}\mathbb{E}\Vert g(\tau,y_{2}(\tau))-g(\tau,y_{1}(\tau))\Vert^{2} \, d\tau\\
			&\leq  4M^{4}N\tilde{L}_{g}\sum_{i=1}^{k}\int_{t_{i-1}}^{t_{i}}\mathbb{E}\Vert y_{2}(\tau)-y_{1}(\tau)\Vert^{2} \, d\tau
				\leq  4M^{4}N\tilde{L}_{g}T\| y_{2}- y_{1} \|^{2}_{PC},
\end{aligned}
\]
where
\begin{align*}
	C_{i}=\prod_{j=k}^{i+1}(1+\Vert D_{j}\Vert)\Vert  T(t_{j}-t_{j-1})\Vert
(1+\Vert D_{i}\Vert),\quad N=\sum_{i=1}^{k}C^{2}_{i}.
\end{align*}

By applying the Itô isometry, C-S inequality and utilizing the  $\mathscr{L}-$ condition, we can derive a following  bound for \(\mathcal{J}_{2}\):
\[
\begin{aligned}
\mathcal{J}_{2}	& =4\mathbb{E}\bigg\Vert  T(t-t_{k})\sum_{i=1}^{k}\prod_{j=k}^{i+1}(\mathcal{I}+D_{j}) T(t_{j}-t_{j-1})
(\mathcal{I}+D_{i})\\
&\times\int_{t_{i-1}}^{t_{i}} T(t_{i}-\tau)(h(\tau,y_{2}(\tau))-h(\tau,y_{1}(\tau)))\, dW(\tau)\bigg\Vert^{2}\\
&\leq  4M^{4}\mathbb{E}\bigg(\sum_{i=1}^{k}\prod_{j=k}^{i+1}(1+\Vert D_{j}\Vert)\Vert  T(t_{j}-t_{j-1})\Vert
(1+\Vert D_{i}\Vert)\\
 &\times\int_{t_{i-1}}^{t_{i}}\Vert h(\tau,y_{2}(\tau))-h(\tau,y_{1}(\tau))\Vert \, dW(\tau)\bigg)^{2}\\
	&\leq  4M^{4}\mathbb{E}\bigg(\sum_{i=1}^{k}C_{i}\int_{t_{i-1}}^{t_{i}}\Vert h(\tau,y_{2}(\tau))-h(\tau,y_{1}(\tau))\Vert \, dW(\tau)\bigg)^{2}
\end{aligned}
\]
\[\begin{aligned}
&\leq  4M^{4}\sum_{i=1}^{k}C^{2}_{i}\sum_{i=1}^{k}\int_{t_{i-1}}^{t_{i}}\mathbb{E}\Vert h(\tau,y_{2}(\tau))-h(\tau,y_{1}(\tau))\Vert^{2} \, d\tau\\
&\leq  4M^{4}N\tilde{L}_{h}\sum_{i=1}^{k}\int_{t_{i-1}}^{t_{i}}\mathbb{E}\Vert y_{2}(\tau)-y_{1}(\tau)\Vert^{2} \, d\tau\\
&\leq  4M^{4}N\tilde{L}_{h}T\| y_{2} - y_{1} \|^{2}_{PC}\\
\end{aligned}
\]

Similarly, by using the appropriate mathematical tools, we obtain the following results for \(\mathcal{J}_{3}\) and \(\mathcal{J}_{4}\):
\[
\begin{aligned}
\mathcal{J}_{3}&=4\mathbb{E}\bigg\Vert\int_{t_{k}}^{t} T(t-\tau)(g(\tau,y_{2}(\tau))-g(\tau,y_{1}(\tau)))\, d\tau\bigg\Vert^{2}\\
	&\leq 4 M^{2} \mathbb{E} \bigg( \int_{t_{k}}^{t} \Vert g(\tau,y_{2}(\tau))-g(\tau,y_{1}(\tau)) \Vert \, d\tau \bigg)^{2} \\
		&\leq 4 M^{2}T  \int_{t_{k}}^{t} \mathbb{E}\Vert g(\tau,y_{2}(\tau))-g(\tau,y_{1}(\tau)) \Vert^{2} \, d\tau\\
	&\leq 4 M^{4} T\tilde{L}_{g} \int_{t_{k}}^{t}  \mathbb{E} \| y_{2}(\tau) - y_{1}(\tau) \|^{2} \, d\tau\\
	&\leq 4 M^{4} T^{2} \tilde{L}_{g} \| y_{2} - y_{1} \|^{2}_{PC},
\end{aligned}
\]
\[
\begin{aligned}	\mathcal{J}_{4}&=4\mathbb{E}\bigg\Vert\int_{t_{k}}^{t} T(t-\tau)(h(\tau,y_{2}(\tau))-h(\tau,y_{1}(\tau)))\, dW(\tau)\bigg\Vert^{2}\\
	&\leq 4 M^{2} \mathbb{E} \bigg( \int_{t_{k}}^{t} \Vert h(\tau,y_{2}(\tau))-h(\tau,y_{1}(\tau)) \Vert \, dW(\tau) \bigg)^{2} \\
	&\leq 4 M^{2}T  \int_{t_{k}}^{t} \mathbb{E}\Vert h(\tau,y_{2}(\tau))-h(\tau,y_{1}(\tau)) \Vert^{2} \, d\tau\\
	&\leq 4 M^{4} T\tilde{L}_{h} \int_{t_{k}}^{t}  \mathbb{E} \| y_{2}(\tau) - y_{1}(\tau) \|^{2} \, d\tau\\
	&\leq 4 M^{4} T^{2} \tilde{L}_{h} \| y_{2} - y_{1} \|^{2}_{PC} .
\end{aligned}
\]

Combining the estimates, we have:

\[
\begin{aligned}
	\mathbb{E}\Vert (Fy_{2})(t)-(Fy_{1})(t)\Vert^{2}& \leq4M^{4}(N+T^{2})(\tilde{L}_{g}	+\tilde{L}_{h})\| y_{2} - y_{1} \|^{2}_{PC}
	=k_{2}\| y_{2} - y_{1} \|^{2}_{PC} .
\end{aligned}
\]
Then, we get
\[
\Vert Fy_{2}-Fy_{1}\Vert^{2}_{PC} \leq k \Vert y_{2} - y_{1} \Vert^{2}_{PC},
\]
where

\begin{align*}
	k=\max\{k_{1},k_{2}\}.
\end{align*}
According to \eqref{tt}, it is established that \( F \) acts as a contraction mapping on \( \mathcal{B}_{r_{0}} \). Consequently, applying the well-established contraction mapping principle confirms that \( F \) possesses a sole fixed point within \( \mathcal{B}_{r_{0}} \). This fixed point \( y \in \mathcal{B}_{r_{0}} \) signifies that \( y(t) \) stands as the unique mild solution to the system described by \eqref{eq1}.
\end{proof}
\section{Existence of optimal controls}

In this section of the manuscript, we delve into the existence of optimal controls for a given control problem. We begin by defining the framework and assumptions necessary for our analysis.

Let \( y^{u} \) represent the mild solution of system \eqref{eq1} associated with the control \( u \in U_{ad} \). Consider the Lagrange problem (P):

Our goal is to find an optimal pair \( (y^{0}, u^{0}) \in PC([0,T],L^{2}(\Omega,H)) \times U_{ad} \) s. t.
\begin{align}
	J(y^{0},u^{0}) \leq J(y^{u},u), \, \forall (y^{u},u) \in PC([0,T],L^{2}(\Omega,H)) \times U_{ad},
\end{align}
where the cost function is defined as
\begin{align*}
	J(y^{u},u) = \mathbb{E}\bigg(\int_{0}^{T} l(t,y^{u}(t),u(t)) \, dt\bigg).
\end{align*}

Suppose the following assumptions hold:

\((A_{1})\) The functional \( l: [0,T] \times H \times U \to \mathbb{R} \cup \{\infty\} \) is \( F_{t} \)-measurable.

\((A_{2})\) For any \( t \in [0,T] \), \( l(t,\cdot,\cdot) \) is sequentially lower semicontinuous on \( H \times U \).

\((A_{3})\) For any \( t \in [0,T] \) and \( x \in H \), \( l(t,y,\cdot) \) is convex on \( U \).

\((A_{4})\) There exist constants \( d_{1} \geq 0 \), \( d_{2} > 0 \), and a nonnegative function \(\xi \in L^{1}([0,T], \mathbb{R})\) s. t.
\begin{align*}
	l(t,y,u) \geq \xi(t) + d_{1} \mathbb{E}\|y\|^{2} + d_{2} \mathbb{E}\|u\|^{2}.
\end{align*}

With these assumptions in place, we are now in a position to demonstrate the existence of optimal controls for problem (P).
\bigskip

\begin{theorem}\label{t3}
	Assume that the hypothesis of Theorem \ref{ty} and assumptions \((A_1)-(A_4)\) are satisfied. Then the Lagrange problem \((P)\) has at least one optimal solution that is, there is an admissible state-control pair
	\begin{align*}
		(y^0, u^0) \in PC([0,T], L^2(\Omega, H)) \times U_{ad},
	\end{align*}
	s. t.
	\begin{align}
		J(y^0, u^0) \leq J(y^u, u), \quad \forall (y^u, u) \in PC([0,T], L^2(\Omega, H)) \times U_{ad}.
	\end{align}
\end{theorem}
\begin{proof}
Without compromising generality, we suppose that
\begin{align*}
	\inf\Big\{ J(y^{u},u) \mid u \in U_{ad}\Big\} = \varepsilon < +\infty.
\end{align*}
If this were not the case, there would be nothing to prove. From assumption \((A_{4})\), it follows that \(\varepsilon > -\infty\). Using the definition of the infimum, we can identify a minimizing sequence of feasible pairs \((y^{m}, u^{m}) \in PC([0,T], L^{2}(\Omega, H)) \times U_{ad}\), satisfying
\[
J(y^{m}, u^{m}) \to \varepsilon \quad \text{as} \quad m \to \infty,
\]
where \(y^{m}\) represents a mild solution of the system \eqref{eq1} associated with \(u^{m} \in U_{ad}\).

We observe that the sequence \(\{u^{m}\}\in U_{ad} \) for \( m = 1, 2, \dots \), which implies that $\{u^{m}\}\in L^{2}_{F}([0,T],U)$ is bounded.  Consequently, there is a function \( u^{0} \in L^{2}_{F}([0,T],U) \) and a subsequence of \(\{u^{m}\}\) s. t.
\[
u^{m} \rightarrow u^{0} \quad (m \to \infty).
\]
Since \( U_{ad} \) is both convex and closed, by the Marzur theorem \cite{11}, we infer that \( u^{0} \in U_{ad} \).

Let \( y^0 \) denote the mild solution of equation \eqref{eq1} corresponding to \( u^0 \). Given the boundedness of \(\{u^m\}\), \(\{u^0\}\), we can assert the existence of a positive number \( r_0 \) s. t.
\[
\|y^m\|^{2}_{PC} \leq r_0, \quad \|y^0\|^{2}_{PC} \leq r_0.
\]
 For \( t \in [t_{0},t_{1}] \), we  obtain
 \begin{align*}
 	\mathbb{E}\Vert y^{m}(t)-y^{0}(t)\Vert^{2}&\leq 3\mathbb{E}\bigg\Vert \int_{0}^{t}  T(t-\tau)B\big(u^{m}(\tau)-u^{0}(\tau)\big)\, d\tau\bigg\Vert^{2}\\
 	&+ 3\mathbb{E}\bigg\Vert \int_{0}^{t}  T(t-\tau)\big(g(\tau,y^{m}(\tau))-g(\tau,y^{0}(\tau))\big)\, d\tau\bigg\Vert^{2}\\
    &+ 3\mathbb{E}\bigg\Vert \int_{0}^{t}  T(t-\tau)\big(h(\tau,y^{m}(\tau))-h(\tau,y^{0}(\tau))\big)\,dW(\tau)\bigg\Vert^{2}\\
 	&\leq 3M^{2}\Vert B\Vert^{2}\int_{0}^{t} \mathbb{E} \Vert u^{m}(\tau)-u^{0}(\tau)\Vert^{2} \, d\tau\\
    &+3M^{2}\mathbb{E}\bigg(\int_{0}^{t} \Vert g(\tau,y^{m}(\tau))-g(\tau,y^{0}(\tau))\Vert \, d\tau\bigg)^{2}\\
 	&+ 3M^{2}\mathbb{E}\bigg(\int_{0}^{t} \Vert h(\tau,y^{m}(\tau))-h(\tau,y^{0}(\tau))\Vert \, dW(\tau)\bigg)^{2}\\
 	&\leq 3M^{2}\Vert B\Vert^{2} \Vert u^{m}-u^{0}\Vert^{2}_{L^{2}_{F}([0,T],U)}\\
 	&+ 3M^{2}T\int_{0}^{t} \mathbb{E}\Vert g(\tau,y^{m}(\tau))-g(\tau,y^{0}(\tau))\Vert^{2} \, d\tau\\
 	&+ 3M^{2}\int_{0}^{t} \mathbb{E}\Vert h(\tau,y^{m}(\tau))-h(\tau,y^{0}(\tau))\Vert^{2} \, d\tau\\
 	&\leq 3M^{2}\Vert B\Vert^{2} \Vert u^{m}-u^{0}\Vert^{2}_{L^{2}_{F}([0,T],U)}\\
 	&+ 3M^{2}(T^{2} \tilde{L}_{g}+T \tilde{L}_{h})\Vert y^{m}-y^{0}\Vert^{2}_{PC},
 \end{align*}
which means
\begin{align}\label{y1}
	\Vert y^{m}-y^{0}\Vert^{2}_{PC}&\leq \frac{3M^{2}\Vert B\Vert^{2} \Vert u^{m}-u^{0}\Vert^{2}_{L^{2}_{F}([0,T],U)}}{1-3M^{2}(T^{2} \tilde{L}_{g}+T \tilde{L}_{h})}.
\end{align}

For $t_k < t \leq t_{k+1}, \, k \geq 1,$  by using the  $\mathscr{L}-$ continuity of \( h \) and the triangular inequality and C-S inequality, we have

\[\begin{aligned}
	&\mathbb{E}\Vert y^{m}(t)-y^{0}(t)\Vert^{2}\\	\leq& 6\mathbb{E}\bigg\Vert  T(t-t_{k})\sum_{i=1}^{k}\prod_{j=k}^{i+1}(\mathcal{I}+D_{j}) T(t_{j}-t_{j-1})
	(\mathcal{I}+D_{i})\int_{t_{i-1}}^{t_{i}} T(t_{i}-\tau)\big(g(\tau,y^{m}(\tau))-g(\tau,y^{0}(\tau))\big) \, d\tau\bigg\Vert^{2}\\
	+&6\mathbb{E}\bigg\Vert  T(t-t_{k})\sum_{i=1}^{k}\prod_{j=k}^{i+1}(\mathcal{I}+D_{j}) T(t_{j}-t_{j-1})
	(\mathcal{I}+D_{i})\int_{t_{i-1}}^{t_{i}} T(t_{i}-\tau)\big(h(\tau,y^{m}(\tau))-h(\tau,y^{0}(\tau))\big) \, dW(\tau)\bigg\Vert^{2}\\
    +&6\mathbb{E}\bigg\Vert \int_{t_{k}}^{t}  T(t-\tau)\big(g(\tau,y^{m}(\tau))-g(\tau,y^{0}(\tau))\big)\, d\tau\bigg\Vert^{2}
	+ 6\mathbb{E}\bigg\Vert \int_{t_{k}}^{t}  T(t-\tau)\big(h(\tau,y^{m}(\tau))-h(\tau,y^{0}(\tau))\big)\, dW(\tau)\bigg\Vert^{2}\\
     +&  6\mathbb{E}\bigg\Vert  T(t-t_{k})\sum_{i=1}^{k}\prod_{j=k}^{i+1}(\mathcal{I}+D_{j}) T(t_{j}-t_{j-1})
	(\mathcal{I}+D_{i})\int_{t_{i-1}}^{t_{i}} T(t_{i}-\tau)B\big(u^{m}(\tau)-u^{0}(\tau)\big) \, d\tau\bigg\Vert^{2}\\
	+&6\mathbb{E}\bigg\Vert \int_{t_{k}}^{t}  T(t-\tau)B\big(u^{m}(\tau)-u^{0}(\tau)\big)\, d\tau\bigg\Vert^{2}
=\mathscr{I}_{1}+\mathscr{I}_{2}+\mathscr{I}_{3}+\mathscr{I}_{4}+\mathscr{I}_5 +\mathscr{I}_6.
\end{aligned}\]

Using the $\mathscr{L}-$ continuity of \( f \) and the triangular inequality and $C-S$- inequality, we get:

\[
\begin{aligned}
	\mathscr{I}_{1}&= 6\mathbb{E}\bigg\Vert  T(t-t_{k})\sum_{i=1}^{k}\prod_{j=k}^{i+1}(\mathcal{I}+D_{j}) T(t_{j}-t_{j-1})
	(\mathcal{I}+D_{i})
	 \int_{t_{i-1}}^{t_{i}} T(t_{i}-\tau)\big(g(\tau,y^{m}(\tau))-g(\tau,y^{0}(\tau))\big) \, d\tau\bigg\Vert^{2}\\
	&\leq 6 M^{2} \mathbb{E}\bigg(\sum_{i=1}^{k}\prod_{j=k}^{i+1}(1+\Vert D_{j}\Vert)\Vert  T(t_{j}-t_{j-1})\Vert
	(1+\Vert D_{i}\Vert)
    \int_{t_{i-1}}^{t_{i}}\Vert g(\tau,y^{m}(\tau))-g(\tau,y^{0}(\tau))\Vert \, d\tau\bigg)^{2}\\
    &=6 M^{2} \mathbb{E}\bigg(\sum_{i=1}^{k}C_{i} \int_{t_{i-1}}^{t_{i}}\Vert g(\tau,y^{m}(\tau))-g(\tau,y^{0}(\tau))\Vert \, d\tau\bigg)^{2}\\
     &\leq 6 M^{2}\sum_{i=1}^{k}C^{2}_{i} \sum_{i=1}^{k}\mathbb{E}\bigg(\int_{t_{i-1}}^{t_{i}}\Vert g(\tau,y^{m}(\tau))-g(\tau,y^{0}(\tau))\Vert \, d\tau\bigg)^{2} \\
     &\leq 6 M^{2}T\sum_{i=1}^{k}C^{2}_{i} \sum_{i=1}^{k}\int_{t_{i-1}}^{t_{i}}\mathbb{E}\Vert g(\tau,y^{m}(\tau))-g(\tau,y^{0}(\tau))\Vert^{2} \, d\tau \\
			&\leq 6 M^{2}TN \sum_{i=1}^{k}\int_{t_{i-1}}^{t_{i}}\tilde{L}_{g}\mathbb{E}\Vert y^{m}(\tau)-y^{0}(\tau)\Vert^{2} \, d\tau \\
			&\leq 6 M^{2}TN\tilde{L}_{g} \sum_{i=1}^{k}\int_{t_{i-1}}^{t_{i}}\mathbb{E}\Vert y^{m}(\tau)-y^{0}(\tau)\Vert^{2} \, d\tau \\
&\leq 6 M^{2}T^{2}N\tilde{L}_{g}\Vert y^{m}-y^{0}\Vert^{2}_{PC},
\end{aligned}
\]
where \begin{align*}
	C_{i}=\prod_{j=k}^{i+1}(1+\Vert D_{j}\Vert)\Vert  T(t_{j}-t_{j-1})\Vert
	(1+\Vert D_{i}\Vert),\quad	N=\sum_{i=1}^{k}C^{2}_{i}.
\end{align*}

Using Itô isometry and the  $\mathscr{L}-$ continuity of \( h \), we obtain:

\[
\begin{aligned}
\mathscr{I}_{2}	&=6\mathbb{E}\bigg\Vert  T(t-t_{k})\sum_{i=1}^{k}\prod_{j=k}^{i+1}(\mathcal{I}+D_{j}) T(t_{j}-t_{j-1})
(\mathcal{I}+D_{i})\int_{t_{i-1}}^{t_{i}} T(t_{i}-\tau)\big(h(\tau,y^{m}(\tau))-h(\tau,y^{0}(\tau))\big) \, dW(\tau)\bigg\Vert^{2}\\
	&\leq 6 M^{2} \mathbb{E}\bigg(\sum_{i=1}^{k}\prod_{j=k}^{i+1}(1+\Vert D_{j}\Vert)\Vert  T(t_{j}-t_{j-1})\Vert
(1+\Vert D_{i}\Vert)\int_{t_{i-1}}^{t_{i}}\Vert h(\tau,y^{m}(\tau))-h(\tau,y^{0}(\tau))\Vert \, dW(\tau)\bigg)^{2}\\
&=6 M^{2} \mathbb{E}\bigg(\sum_{i=1}^{k}C_{i} \int_{t_{i-1}}^{t_{i}}\Vert h(\tau,y^{m}(\tau))-h(\tau,y^{0}(\tau))\Vert \, dW(\tau)\bigg)^{2}\\
 &\leq 6 M^{2}\sum_{i=1}^{k}C^{2}_{i} \sum_{i=1}^{k}\mathbb{E}\bigg(\int_{t_{i-1}}^{t_{i}}\Vert h(\tau,y^{m}(\tau))-h(\tau,y^{0}(\tau))\Vert \, dW(\tau)\bigg)^{2} \\
&\leq 6 M^{2}\sum_{i=1}^{k}C^{2}_{i} \sum_{i=1}^{k}\int_{t_{i-1}}^{t_{i}}\mathbb{E}\Vert h(\tau,y^{m}(\tau))-h(\tau,y^{0}(\tau))\Vert^{2} \, d\tau \leq 6 M^{2}N \sum_{i=1}^{k}\int_{t_{i-1}}^{t_{i}}\tilde{L}_{h}\mathbb{E}\Vert y^{m}(\tau)-y^{0}(\tau)\Vert^{2} \, d\tau \\
&\leq 6 M^{2}N\tilde{L}_{h} \sum_{i=1}^{k}\int_{t_{i-1}}^{t_{i}}\mathbb{E}\Vert y^{m}(\tau)-y^{0}(\tau)\Vert^{2} \, d\tau \leq 6 M^{2}TN\tilde{L}_{h}\Vert y^{m}-y^{0}\Vert^{2}_{PC}.
\end{aligned}
\]
\bigskip

Using Itô isometry and the  $\mathscr{L}-$ continuity of \( f \) and \( h \):

\[
\begin{aligned}
	\mathscr{I}_{3}	&=6\mathbb{E}\bigg\Vert \int_{t_{k}}^{t}  T(t-\tau)\big(g(\tau,y^{m}(\tau))-g(\tau,y^{0}(\tau))\big)\, d\tau\bigg\Vert^{2}\\
	&\leq 6M^{2}T\int_{t_{k}}^{t} \mathbb{E}\Vert g(\tau,y^{m}(\tau))-g(\tau,y^{0}(\tau))\Vert^{2} \, d\tau\\
	&\leq 6M^{2}T\int_{t_{k}}^{t} \tilde{L}_{g}\mathbb{E}\Vert y^{m}(\tau)-y^{0}(\tau)\Vert^{2} \, d\tau\\
		&\leq6 M^{2}T^{2}\tilde{L}_{g}\Vert y^{m}-y^{0}\Vert^{2}_{PC},
\end{aligned}
\]

\[
\begin{aligned}
	\mathscr{I}_{4}	&=6\mathbb{E}\bigg\Vert \int_{t_{k}}^{t}  T(t-\tau)\big(h(\tau,y^{m}(\tau))-h(\tau,y^{0}(\tau))\big)\, dW(\tau)\bigg\Vert^{2}\\
	&\leq 6M^{2}\int_{t_{k}}^{t} \mathbb{E}\Vert h(\tau,y^{m}(\tau))-h(\tau,y^{0}(\tau))\Vert^{2} \, d\tau\\
		&\leq6 M^{2}\int_{t_{k}}^{t} \tilde{L}_{h}\mathbb{E}\Vert y^{m}(\tau)-y^{0}(\tau)\Vert^{2} \, d\tau\\
	&\leq6 M^{2}T\tilde{L}_{h}\Vert y^{m}-y^{0}\Vert^{2}_{PC}.
\end{aligned}
\]
\bigskip

Following a similar process as described above, we have

\[\begin{aligned}
\mathscr{I}_{5}&=6\mathbb{E}\bigg\Vert  T(t-t_{k})\sum_{i=1}^{k}\prod_{j=k}^{i+1}(\mathcal{I}+D_{j}) T(t_{j}-t_{j-1})
	(\mathcal{I}+D_{i})
	\int_{t_{i-1}}^{t_{i}} T(t_{i}-\tau)B\big(u^{m}(\tau)-u^{0}(\tau)\big) \, d\tau\bigg\Vert^{2}\\
	&\leq 6 M^{2}\Vert B\Vert^{2} \mathbb{E}\bigg(\sum_{i=1}^{k}\prod_{j=k}^{i+1}(1+\Vert D_{j}\Vert)\Vert  T(t_{j}-t_{j-1})\Vert
	(1+\Vert D_{i}\Vert)
	 \int_{t_{i-1}}^{t_{i}}\Vert u^{m}(\tau)-u^{0}(\tau)\Vert \, d\tau\bigg)^{2}\\
	&=6 M^{2}\Vert B\Vert^{2} \mathbb{E}\bigg(\sum_{i=1}^{k}C_{i} \int_{t_{i-1}}^{t_{i}}\Vert  u^{m}(\tau)-u^{0}(\tau)\Vert \, d\tau\bigg)^{2}\\
	&\leq 6 M^{2}\Vert B\Vert^{2}\sum_{i=1}^{k}C^{2}_{i} \sum_{i=1}^{k}\mathbb{E}\bigg(\int_{t_{i-1}}^{t_{i}}\Vert  u^{m}(\tau)-u^{0}(\tau)\Vert \, d\tau\bigg)^{2} \\
	&\leq 6 M^{2}\Vert B\Vert^{2}T\sum_{i=1}^{k}C^{2}_{i} \sum_{i=1}^{k}\int_{t_{i-1}}^{t_{i}}\mathbb{E}\Vert u^{m}(\tau)-u^{0}(\tau)\Vert^{2} \, d\tau \\
	&= 6 M^{2}\Vert B\Vert^{2}TN \sum_{i=1}^{k}\int_{t_{i-1}}^{t_{i}}\mathbb{E}\Vert u^{m}(\tau)-u^{0}(\tau)\Vert^{2} \, d\tau \\
	&= 6 M^{2}\Vert B\Vert^{2}TN\Vert u^{m}-u^{0}\Vert^{2}_{L^{2}_{F}([0,T],U)},
\end{aligned}\]
\begin{align*}
	\mathscr{I}_{6}&=6\mathbb{E}\bigg\Vert \int_{t_{k}}^{t}  T(t-\tau)B\big(u^{m}(\tau)-u^{0}(\tau)\big)\, d\tau\bigg\Vert^{2}\\
	&\leq 6M^{2}\Vert B\Vert^{2}T\int_{t_{k}}^{t} \mathbb{E}\Vert u^{m}(\tau)-u^{0})\Vert^{2} \, d\tau\\
&\leq 6M^{2}\Vert B\Vert^{2}T\Vert u^{m}-u^{0}\Vert^{2}_{L^{2}_{F}([0,T],U)}.
\end{align*}
Combining all the estimates, we achieve:

\[
\begin{aligned}
\mathbb{E}\Vert y^{m}(t)-y^{0}(t)\Vert^{2} &\leq 6 M^{2}(N+1)(T^{2} \tilde{L}_{g}+T\tilde{L}_{h})\Vert y^{m}-y^{0}\Vert^{2}_{PC}+6 M^{2}T(N+1)\Vert B\Vert^{2}\Vert u^{m}-u^{0}\Vert^{2}_{L^{2}_{F}([0,T],U)},
\end{aligned}
\]
which means
\begin{align}\label{y2}
	\Vert y^{m}-y^{0}\Vert^{2}_{PC}&\leq \frac{6 M^{2}T(N+1)\Vert B\Vert^{2}\Vert u^{m}-u^{0}\Vert^{2}_{L^{2}_{F}([0,T],U)}}{1-6 M^{2}(N+1)(T^{2} \tilde{L}_{g}+T\tilde{L}_{h})}.
\end{align}
Using \eqref{y1} and \eqref{y2}, we get
\begin{align}
	\Vert y^{m}-y^{0}\Vert^{2}_{PC}&\leq C^{*}\Vert u^{m}-u^{0}\Vert^{2}_{L^{2}_{F}([0,T],U)},
\end{align}
where
\begin{align*}
	C^{*}&=\max\Bigg\{ \frac{3M^{2}\Vert B\Vert^{2} }{1-3M^{2}(T^{2}\tilde{L}_{g}+T\tilde{L}_{h})} ,\frac{6 M^{2}T(N+1)\Vert B\Vert^{2}}{1-6 M^{2}(N+1)(T^{2} \tilde{L}_{g}+T\tilde{L}_{h})}\Bigg\}.
\end{align*}
Since,
\begin{align*}
	\Vert u^{m}-u^{0}\Vert^{2}_{L^{2}_{F}([0,T],U)}\xrightarrow{s} 0,\quad (m\to \infty).
\end{align*}
As a result,
\begin{align*}
	\Vert y^{m}-y^{0}\Vert^{2}_{PC}\xrightarrow{s} 0,\quad (m\to \infty).
\end{align*}
Thus, by leveraging conditions \((A_1)-(A_4)\) together with Balder's theorem (see Theorem 2.1 in \cite{13}), it follows that the mapping \((y, u) \mapsto \mathbb{E}\left(\int_0^T l(t, y(t), u(t)) \, dt\right)\) is sequentially lower semicontinuous. This holds with respect to the strong topology of \(L^1_F([0,T], H)\) and the weak topology of \(L^2_F([0,T], U) \subset L^1_F([0,T], U)\). Consequently, the functional \(J\) is weakly lower semicontinuous on \(L^2_F([0,T], U)\). Therefore, we have:
\begin{align*}
	\varepsilon &= \lim_{m \to \infty} \mathbb{E}\left(\int_0^T l(t, y^m(t), u^m(t)) \, dt\right) \geq \mathbb{E}\left(\int_0^T l(t, y^0(t), u^0(t)) \, dt\right) = J(y^0, u^0) \geq \varepsilon,
\end{align*}
which implies that \(u^0 \in U_{ad}\) is a minimizer of \(J\).
\end{proof}

\section{Application}

Consider the stochastic impulsive system described by:

\begin{equation}\label{pp}
	\begin{cases}
		y'(t) = A y(t) + B u(t) + g(t, y(t)) + h(t, y(t)) \frac{dW(t)}{dt}, & t \in [0,1] \setminus \{ \frac{1}{2} \}, \\
		\Delta y\left(\frac{1}{2}\right) = D_{1} y\left(\frac{1}{2}\right) + E_{1} v_{1}, \\
		y(0) = 0.
	\end{cases}
\end{equation}

where \(H = U = L^2[0,1]\).

Define the operators and functions as follows:

\begin{align}
	\begin{cases}
		A  = -\frac{d^{2}}{d t^{2}} + \frac{d}{dt}, \\
		B = D_{1} = E_{1} = \mathcal{I}, \\
		g(t, y(t)) = \frac{2}{5} \cos(t) + \frac{y(t)}{t + 5}, \\
		h(t, y(t)) = \frac{1}{5} \left(\frac{2}{1 + e^t} + \frac{|y(t)|}{1 + |y(t)|}\right), \\
		v_{1} = \sin(\pi t).
	\end{cases}
\end{align}

Here, \(A\) is defined as an operator \(A: D(A) \subset L^2[0,1] \to L^2[0,1]\) with

\begin{align*}
	D(A) = \{y \in L^2[0,1] : y(0) = y(1) = 0\}.
\end{align*}

Substituting these definitions into \eqref{pp}, we receive:

\begin{equation}\label{ppp2}
	\begin{cases}
		y^{\prime\prime}(t) =  u(t) + \frac{2}{5} \cos(t) + \frac{y(t)}{t + 5} + \frac{1}{5} \left(\frac{2}{1 + e^t} + \frac{|y(t)|}{1 + |y(t)|}\right) \frac{dW(t)}{dt}, & t \in [0,1] \setminus \left\{\frac{1}{2}\right\}, \\
		y\left(\frac{1}{2}^+\right) = y\left(\frac{1}{2}^-\right) +  y\left(\frac{1}{2}\right) + \sin(\pi t), \\
		y(0) = 0.
	\end{cases}
\end{equation}

The problem \eqref{ppp2} can be framed in the abstract form of \eqref{eq1}, with the cost functional given by:

\begin{align*}
	J(y,u) = \mathbb{E} \int_{0}^{1}\big( \| y \|^2 + \| u \|^2 \big)  dt.
\end{align*}

The functions \(g\) and \(\sigma\) are shown to satisfy the  $\mathscr{L}-$ conditions specified in Assumptions \ref{assump1} and \ref{assump2}, with parameters \(L_g = L_h = \frac{2}{5}\) and \(\tilde{L}_g = \tilde{L}_h = \frac{1}{25}\).

Given that all assumptions \((A_1)-(A_4)\) are satisfied, Theorem \ref{t3} ensures the existence of at least one optimal pair \((y^0, u^0)\) s. t. :

\[ J(y^0, u^0) \leq J(y^u, u), \quad \forall (y^u, u) \in PC([0,T], L^2(\Omega, H)) \times U_{ad}. \]

Thus, Theorem \ref{t3} is applicable to the given stochastic impulsive system, guaranteeing the existence of an optimal solution to the problem.

\section{Conclusion}

This study delves into the solvability and optimal control of a class of  SDEs within a Hilbert space framework with impulsive effects. Specifically, we established the existence/uniqueness of mild solutions for the system by employing fixed-point theorems and deriving sufficient conditions for the existence of optimal control pairs. These results ensure both the feasibility and optimality of the proposed control strategies. Furthermore, to demonstrate the practical relevance and effectiveness of our theoretical findings, we presented a comprehensive example showcasing their application.

Beyond addressing the immediate problems of existence and control, this research contributes to a deeper understanding of impulsive stochastic systems and their dynamic behavior under stochastic influences. It highlights potential extensions to more complex scenarios, including systems with state-dependent impulses or infinite delays. The work serves as a robust foundation for future studies aiming to refine control strategies, extend the theoretical framework, or explore applications in fields such as engineering, finance, and biological modeling.
Here are some key features of impulsive SDEs: 
\begin{itemize}
    \item Continuous evolution governed by the SDE between impulses.
    \item Discrete, instantaneous changes at predefined times or state-dependent moments.
    \item Can model real-world phenomena like control interventions, economic shocks, or biological mutations.
    \item The impulsive function \( I_k \) can depend on time, the state, or both.
    \item The stochastic term accounts for randomness in continuous evolution.
    \item Impulses introduce deterministic or stochastic discontinuities.
\end{itemize}
In the future, the controllability of impulsive SDEs will continue to be a challenging and highly relevant area of our research. Advancements in both analytical methods and computational techniques are expected to provide deeper insights into the interaction between continuous dynamics, randomness, and discrete impulsive interventions. As systems with impulsive effects become more common in fields like engineering, finance, and biology, it will be crucial to develop more sophisticated control strategies that account for both the stochastic nature of the system and the instantaneous changes due to impulses. The integration of machine learning, optimization algorithms, and advanced numerical solvers will likely play a pivotal role in designing efficient controllers and solving complex controllability problems. Additionally, future studies may focus on expanding the applicability of these methods to more general classes of impulsive SDEs and real-world scenarios, pushing the boundaries of control theory to handle increasingly complex, real-time systems.


\begin{thebibliography}{99}%
	\bibitem{1} Mainardi, F., Paradisi, P., \& Gorenflo, R. (2007). Probability distributions generated by fractional diffusion equations. arXiv preprint arXiv:0704.0320.
	\bibitem{2} Mahmudov, N. I. (2024). A study on approximate controllability of linear impulsive equations in Hilbert spaces. Quaestiones Mathematicae, 1-16.
	\bibitem{3} Periago, F., \& Straub, B. (2002). A functional calculus for almost sectorial operators and applications to abstract evolution equations. Journal of Evolution Equations, 2(1), 41-68.
	\bibitem{4} Asadzade, J.A., \& Mahmudov, N.I. (2024). Approximate Controllability of Linear Fractional Impulsive Evolution Equations in Hilbert Spaces. arXiv preprint  arXiv:2406.15114
	\bibitem{5} Asadzade, J. A., \& Mahmudov, N. I. (2024). Finite time stability analysis for fractional stochastic neutral delay differential equations. Journal of Applied Mathematics and Computing, 1-25.
	\bibitem{6}  Barraez, D., Leiva, H., Merentes, N., \& Narváez, M. (2011). Exact controllability of semilinear stochastic evolution equation.
	\bibitem{7} Kilbas, A. A., Srivastava, H. M., \& Trujillo, J. J. (2006). Theory and applications of fractional differential equations (Vol. 204). elsevier.
	\bibitem{8} Zhou, Y., \& Jiao, F. (2010). Nonlocal Cauchy problem for fractional evolution equations. Nonlinear analysis: real world applications, 11(5), 4465-4475.
	\bibitem{9} Da Prato, G., \& Zabczyk, J. (2014). Stochastic equations in infinite dimensions (Vol. 152). Cambridge university press.
	\bibitem{10} Mao, X. R. (1997). Stochastic differential equations and their applications, horwood publ. House, Chichester.
	\bibitem{11} Heinonen, J., Kipelainen, T., \& Martio, O. (2018). Nonlinear potential theory of degenerate elliptic equations. Courier Dover Publications.
	\bibitem{12}  Ding, Y., \& Niu, J. (2024). Solvability and optimal controls of fractional impulsive stochastic evolution equations with nonlocal conditions. Journal of Applied Analysis \& Computation, 14(5), 2622-2642.
	\bibitem{13} Balder, E. J. (1987). Necessary and sufficient conditions for L1-strong-weak lower semicontinuity of integral functionals. Nonlinear Analysis: Theory, Methods \& Applications, 11(12), 1399-1404.
	\bibitem{14} Dhayal, R., Malik, M., \& Abbas, S. (2021). Solvability and optimal controls of non-instantaneous impulsive stochastic fractional differential equation of order $q\in(1, 2)$. Stochastics, 93(5), 780-802.
	\bibitem{15} Bainov, D., \& Simeonov, P. (2017). Impulsive differential equations: periodic solutions and applications. Routledge.
	\bibitem{16} Balasubramaniam, P., \& Tamilalagan, P. (2017). The solvability and optimal controls for impulsive fractional stochastic integro-differential equations via resolvent operators. Journal of Optimization Theory and Applications, 174, 139-155.
	\bibitem{17} Benchohra, M. (2006). Impulsive differential equations and inclusions (Vol. 2). Hindawi Publishing Corporation.
	\bibitem{18} Byszewski, L. (1991). Theorems about the existence and uniqueness of solutions of a semilinear evolution nonlocal Cauchy problem. Journal of Mathematical analysis and Applications, 162(2), 494-505.
	\bibitem{19} Yan, Z. (2021). Time optimal control of a Clarke subdifferential type stochastic evolution inclusion in Hilbert spaces. Applied Mathematics \& Optimization, 84(3), 3083-3110.
	\bibitem{20} Mahmudov, N. I., Vijayakumar, V., \& Murugesu, R. (2016). Approximate controllability of second-order evolution differential inclusions in Hilbert spaces. Mediterranean Journal of Mathematics, 13, 3433-3454.
	\bibitem{21} Anukiruthika, K., Durga, N., \& Muthukumar, P. (2023). Optimal control of non-instantaneous impulsive second-order stochastic McKean–Vlasov evolution system with Clarke subdifferential. International Journal of Nonlinear Sciences and Numerical Simulation, 24(6), 2061-2087.
	\bibitem{22} Zhou, J., \& Liu, B. (2010). Optimal control problem for stochastic evolution equations in Hilbert spaces. International journal of control, 83(9), 1771-1784.
	\bibitem{23} Al-Hussein, A. (2011). Necessary conditions for optimal control of stochastic evolution equations in Hilbert spaces. Applied Mathematics \& Optimization, 63, 385-400.
	\bibitem{24} Mahmudov, N. I., \& McKibben, M. A. (2007). On backward stochastic evolution equations in Hilbert space and optimal control. Nonlinear Analysis Series A: Theory, Methods, and Applications, 67(4), 1262.
	\bibitem{25} Subalakshmi, R., \& Radhakrishnan, B. (2021). A study on approximate and exact controllability of impulsive stochastic neutral integrodifferential evolution system in Hilbert spaces. International Journal of Nonlinear Analysis and Applications, 12(Special Issue), 1731-1743.
	\bibitem{26} Levajković, T., Mena, H., \& Tuffaha, A. (2016). The stochastic linear quadratic optimal control problem in Hilbert spaces: A polynomial chaos approach. Evol. Equ. Control Theory, 5(1), 105-134.
	\bibitem{27} Chadha, A., \& Bora, S. N. (2021). Solvability of control problem for a  nonlocal neutral stochastic fractional integro-differential inclusion with impulses. Mathematical Reports, 23(3), 265-294.
	\bibitem{28} Carmichael, N., \& Quinn, M. D. (1985). Fixed point methods in nonlinear control. In Distributed Parameter Systems: Proceedings of the 2nd International Conference Vorau, Austria 1984 (pp. 24-51). Springer Berlin Heidelberg.
	\bibitem{29} Subalakshmi, R., \& Balachandran, K. (2009). Approximate controllability of nonlinear stochastic impulsive integrodifferential systems in Hilbert spaces. Chaos, Solitons \& Fractals, 42(4), 2035-2046.
	\bibitem{30} Mahmudov, N. I. (2003). Approximate controllability of semilinear deterministic and stochastic evolution equations in abstract spaces. SIAM journal on control and optimization, 42(5), 1604-1622.
	\bibitem{31} Ahmadova, A. (2023). Approximate controllability of stochastic degenerate evolution equations: Decomposition of a Hilbert space. Differential Equations and Dynamical Systems, 1-23.
	\bibitem{32} Sakthivel, R., Suganya, S., \& Anthoni, S. M. (2012). Approximate controllability of fractional stochastic evolution equations. Computers \& Mathematics with Applications, 63(3), 660-668.
	\bibitem{33} Asadzade, J. A., \& Mahmudov, N. I. (2024). Euler-Maruyama approximation for stochastic fractional neutral integro-differential equations with weakly singular kernel. Physica Scripta, 99(7), 075281.
	\bibitem{34} Bashirov, A. E., \& Mahmudov, N. I. (1999). On concepts of controllability for deterministic and stochastic systems. SIAM Journal on Control and Optimization, 37(6), 1808-1821.
	\bibitem{35} Chang, Y. K., Pei, Y., \& Ponce, R. (2019). Existence and optimal controls for fractional stochastic evolution equations of Sobolev type via fractional resolvent operators. Journal of Optimization Theory and Applications, 182, 558-572.
	\bibitem{36} Yan, Z., \& Jia, X. (2017). Optimal controls of fractional impulsive partial neutral stochastic integro-differential systems with infinite delay in Hilbert spaces. International Journal of Control, Automation and Systems, 15(3), 1051-1068.
	\bibitem{37} Li, X., \& Liu, Z. (2015). The solvability and optimal controls of impulsive fractional semilinear differential equations.
	\bibitem{38} Yan, Z., \& Lu, F. (2018). Solvability and optimal controls of a fractional impulsive stochastic partial integro-differential equation with state-dependent delay. Acta Applicandae Mathematicae, 155(1), 57-84.
	\bibitem{39} Sheng, L., Hu, W., \& Su, Y. H. (2024). Existence and optimal controls of non-autonomous for impulsive evolution equation without Lipschitz assumption. Boundary Value Problems, 2024(1), 17.
	\bibitem{40} Samoilenko, A. M., \& Perestyuk, N. A. (1995). Impulsive differential equations. World scientific.
	\bibitem{41} Lakshmikantham, V., \& Simeonov, P. S. (1989). Theory of impulsive differential
	equations (Vol. 6). World scientific.
	\bibitem{42} Mahmudov, N. I., \& Almatarneh, A. M. (2020). Stability of Ulam–Hyers and existence of solutions for impulsive time-delay semi-linear systems with non-permutable matrices. Mathematics, 8(9), 1493.

  \bibitem{43} Yin, Q. B., Shu, X. B., Guo, Y., \& Wang, Z. Y. (2024). Optimal control of stochastic differential equations with random impulses and the Hamilton–Jacobi–Bellman equation. Optimal Control Applications and Methods, 45(5), 2113-2135.

    \bibitem{44} Guo, Y., Shu, X. B., Xu, F., \& Yang, C. (2024). HJB equation for optimal control system with random impulses. Optimization, 73(4), 1303-1327.

    \bibitem{45} Zhou, B., Shu, X. B., Xu, F., Yang, F., \& Wang, Y. (2024). Exponential synchronization of dynamical complex networks via random impulsive scheme. Nonlinear Analysis: Modelling and Control, 29(4), 816-832.

    \bibitem{46} Asadzade, J. A., \& Mahmudov, N. I. (2024). Approximate controllability of impulsive semilinear evolution equations in Hilbert spaces. arXiv preprint arXiv:2411.02766.
\end{thebibliography}
\end{document}